\documentclass[10pt,leqno,twoside]{amsart}
\usepackage{amsmath}
\usepackage{amsfonts}
\usepackage{enumerate}
\usepackage{amsthm}
\usepackage{bm}
\usepackage{mathtools}
\usepackage{color}
\usepackage{booktabs}
\makeatletter
	 \@addtoreset{equation}{section}
	 \@addtoreset{figure}{section}
	 \@addtoreset{table}{section}
\makeatother
\newcommand{\eref}[1]{(\ref{#1})}
\newcommand{\tref}[1]{Theorem \ref{#1}}
\newcommand{\pref}[1]{Proposition \ref{#1}}
\newcommand{\lref}[1]{Lemma \ref{#1}}
\newcommand{\rref}[1]{Remark \ref{#1}}
\newcommand{\sref}[1]{Section \ref{#1}}
\newcommand{\fref}[1]{Figure \ref{#1}}
\newcommand{\tabref}[1]{Table \ref{#1}}

\theoremstyle{plain} \newtheorem{thm}{Theorem}[section] \newtheorem{lem}{Lemma}[section] \newtheorem{prop}{Proposition}[section] 
\theoremstyle{definition} \newtheorem{rem}{Remark}[section]
\title[Finite element approximation for the Stokes equations]{Penalty method with P1/P1 finite element approximation for the Stokes equations under slip boundary condition}
\author[Takahito Kashiwabara]{Takahito Kashiwabara}
\address{Department of Mathematics, Tokyo Institute of Technology, 2-12-1 Ookayama, Meguro, Tokyo 152-8551, Japan}
\email{tkashiwa@math.titech.ac.jp}
\author[Issei Oikawa]{Issei Oikawa}
\address{Faculty of Science and Engineering, Waseda University, 3-4-1 Okubo, Shinjuku, Tokyo 169-8555, Japan}
\email{oikawa@aoni.waseda.jp}
\author[Guanyu Zhou]{Guanyu Zhou}
\address{Graduate School of Mathematical Sciences, The University of Tokyo, 3-8-1 Komaba, Meguro, Tokyo 153-8914, Japan}
\email{zhoug@ms.u-tokyo.ac.jp}

\begin{document}
\begin{abstract}
	We consider the P1/P1 or P1b/P1 finite element approximations to the Stokes equations in a bounded smooth domain subject to the slip boundary condition.
	A penalty method is applied to address the essential boundary condition $u\cdot n = g$ on $\partial\Omega$, which avoids a variational crime and simultaneously facilitates the numerical implementation.
	We give $O(h^{1/2} + \epsilon^{1/2} +  h/\epsilon^{1/2})$-error estimate for velocity and pressure in the energy norm, where $h$ and $\epsilon$ denote the discretization parameter and the penalty parameter, respectively.
	In the two-dimensional case, it is improved to $O(h + \epsilon^{1/2} +  h^2/\epsilon^{1/2})$ by applying reduced-order numerical integration to the penalty term.
	The theoretical results are confirmed by numerical experiments.
\end{abstract}
\maketitle
\section{Introduction} \label{sec:intro}
In this paper, letting $\Omega\subset\mathbb R^N\, (N=2,3)$ be a bounded smooth domain, we consider the Stokes equations subject to the non-homogeneous slip boundary condition as follows:
\begin{equation} \label{eq:Stokes slip BC}
	\left\{
	\begin{aligned}
		u - \nu \Delta u + \nabla p &= f &\mbox{in} &&\Omega, \\
		\mathrm{div}\,u &= 0 &\mbox{in} &&\Omega, \\
		u\cdot n &= g &\mbox{on} &&\partial\Omega, \\
		(I - n\otimes n) \sigma(u, p)n &= \tau &\mbox{on} &&\partial\Omega,
	\end{aligned}
	\right.
\end{equation}
where $u:\Omega\to\mathbb R^N$ and $p:\Omega\to\mathbb R$ are the velocity and pressure of the fluid respectively, and $\nu>0$ is a viscosity constant.
Moreover, $f$ represents the given body force, $g$ the prescribed outgoing flow on the boundary $\Gamma := \partial\Omega$, and $\tau$ the prescribed traction vector on $\Gamma$ in the tangential direction, with $\sigma(u, p) = -pI + \nu(\nabla u + (\nabla u)^T)$ being the Cauchy stress tensor associated with the fluid.
The outer unit normal to the boundary $\Gamma$ is denoted by $n$.
The first term in \eref{eq:Stokes slip BC}$_1$ is added in order to ensure coercivity of the problem without taking into account rigid body movements.
We impose the compatibility condition between \eref{eq:Stokes slip BC}$_2$ and \eref{eq:Stokes slip BC}$_3$ which reads
\begin{equation} \label{eq:compatibility condition for g}
	\int_\Gamma g\,d\gamma = 0.
\end{equation}

The slip boundary condition \eref{eq:Stokes slip BC}$_3$--\eref{eq:Stokes slip BC}$_4$ (or its variant the Navier boundary condition) is now widely accepted as one of the standard boundary conditions for the Navier-Stokes equations.
There are many applications of the slip boundary conditions to real flow problems; here we only mention the coating problem \cite{SaSc81} and boundary conditions of high Reynolds number flow \cite{Lay99}.
For more details on the application side of slip-type boundary conditions, we refer to Stokes and Carey \cite{StCa11} and references therein; see also John \cite{Joh02} for generalization combined with leak-type boundary conditions.

In the present paper, our motivation to consider problem \eref{eq:Stokes slip BC} consists in dealing with some mathematical difficulties which are specific to its finite element approximation.
As shown by Solonnikov and \v{S}\v{c}adilov in \cite{SoSc73} (see also Beir\~{a}o da Veiga \cite{BdV04} for a generalized non-homogeneous problem), proving the existence, uniqueness, and regularity of a solution of \eref{eq:Stokes slip BC} does not reveal essentially more difficulty compared with the case of Dirichlet boundary conditions.
Then one is led to hope that its finite element approximation could also be treated analogously to the Dirichlet case.
However, it is known that a naive discretization of \eref{eq:Stokes slip BC}, especially when a smoothly curved domain $\Omega$ is approximated by a polyhedral domain $\Omega_h$, leads to a variational crime in which we no longer obtain convergence of approximate solutions.

Let us describe this phenomenon assuming $g=0$ and considering piecewise linear approximation of velocity.
In view of the weak formulation of the continuous problem, see \eref{eq:weak form with constraint} below, a natural choice of the space to approximate velocity would be (we adopt the notation of \sref{sec:finite element approximation}):
\begin{equation} \label{eq:Vhn}
	V_{hn} = \{ v_h\in V_h \,:\, v_h\cdot n_h = 0 \mbox{ on }\Gamma_h \}, \quad \Gamma_h:=\partial\Omega_h,
\end{equation}
where $n_h$ denotes the outer unit normal associated to $\Gamma_h$.
Now suppose that $N=2$ and that any two adjacent edges that constitute $\Gamma_h$ are not parallel.
Then, one readily sees that $V_{hn}$ above reduces to $\mathring V_h = V_h \cap H^1_0(\Omega_h)^2$.
As a result, the finite element solution computed using $V_{hn}$ is nothing but the one satisfying the Dirichlet boundary condition, which completely fails to approximate the slip boundary condition.
For $N=3$ or quadratic approximation (or whatever else), we may well expect similar undesirable reduction in the degrees of freedom that should have been left for the velocity components, which accounts for the variational crime.

One way to overcome the variational crime is to replace the constraint in \eref{eq:Vhn} by $(v_h\cdot n)(P) = 0$ for each boundary node $P$.
This strategy was employed by Tabata and Suzuki \cite{TaSu00} where $\Omega$ is a spherical shell; see also Tabata \cite{Tab01,Tab06}.
Extension of the idea to the quadratic approximation was proposed by B\"ansch and Deckelnick in \cite{BaDe99} using some abstract transformation $G_h: \Omega_h\to\Omega$ introduced by Lenoir \cite{Len86}.
For $\Omega$ of general shape, the exact values for $n(P)$ or $n\circ G_h(P)$ may not be available.
In this regard, some average of $n_h$'s near the boundary node $P$ can be used as approximation of those unavailable values.
This idea was numerically tested by B\"ansch and H\"ohn in \cite{BaHo00} for $N=3$ and by Dione, Tibirna and Urquiza \cite{DTU13} for $N=2$ (with penalty formulation), showing good convergence property.
However, rigorous and systematic evaluation of those approximations is non-trivial and does not seem to be known in the literature.
Moreover, implementation of constraints like $(v_h\cdot n)(P) = 0$ in a real finite element code is also non-trivial and requires special computational techniques (see e.g.\ Gresho and Sani \cite[p.\ 540]{GrSa00} and \cite[Section 5]{BaDe99}), which are not necessary to treat the Dirichlet boundary condition.

In view of these situations, in the present paper we would like to investigate a finite element scheme to \eref{eq:Stokes slip BC} such that: 1) rigorous error analysis can be performed; 2) numerical implementation is as easy as for the Dirichlet case.
With this aim we adopt a \emph{penalty approach} proposed by Dione and Urquiza \cite{DiUr14} (see also \cite{DTU13}) which, in the continuous setting, replaces the Dirichlet condition \eref{eq:Stokes slip BC}$_3$ by the Robin-type one involving a very small number (called the penalty parameter) $\epsilon>0$, i.e., $\sigma(u,p)n\cdot n + \frac1\epsilon (u\cdot n - g) = 0$ on $\Gamma$.
At the weak formulation level, this amounts to removing the constraint $v\cdot n = 0$ from the test function space and introducing a penalty term $\frac1\epsilon (u\cdot n - g, v\cdot n)_\Gamma$ in the weak form.
Our scheme transfers this procedure to the discrete setting given on $\Omega_h$; see \eref{eq:discrete problem with 2 variables} below.
Since the test function space for velocity is taken as the whole $V_h$ involving no constraints, this scheme facilitates implementation, which serves purpose 2) mentioned above.
It is indeed simple enough to be implemented by well-known finite element libraries such as \verb#FreeFem++# \cite{freefem} and \verb+FEniCS+ \cite{fenics}, as is presented in our numerical examples.

Let us turn our attention to the error analysis.
The first error estimate was given by Verf\"urth \cite{Ver85} who derived $O(h^{1/2})$ in the energy norm for $g=\tau=0$.
The same author proposed the Lagrange multiplier approach in \cite{Ver87,Ver91} for $\tau=0$.
Later, Knobloch \cite{Kno99} derived optimal error estimates (namely, $O(h)$ for linear-type approximation and $O(h^{3/2})$ for quadratic-type approximation) for $g=0$ and for various combinations of finite elements satisfying the LBB condition, assuming the existence of better approximation of $n$ than $n_h$.
The convergence (without rate) under minimal regularity assumptions was proved by the same author in \cite{Kno00}.
A different proof of $O(h^{3/2})$-estimate for the P2/P1 element was given by \cite{BaDe99} for $g=\tau=0$, assuming that $n\circ G_h$ is known.
The technique using $G_h$ was then exploited to study the penalty scheme in \cite{DiUr14}, again for the P2/P1 element and for $g=\tau=0$.

In the present paper, we study the penalty scheme for the P1/P1 element combined with pressure stabilization and also for the P1b/P1 element.
Our method to establish the error estimate is quite different from those of the preceding works mentioned above.
First, we address the non-homogeneous boundary conditions \eref{eq:Stokes slip BC}$_3$--\eref{eq:Stokes slip BC}$_4$ which were not considered previously.
Second, concerning the penalty scheme, we directly compare $(u, p)$ and $(u_h, p_h)$, whereas Dione and Urquiza \cite{DiUr14} introduced a penalized problem in the continuous setting, dividing the error estimates into two stages.

Third, we define our error (for velocity) to be $\|\tilde u - u_h\|_{H^1(\Omega_h)}$, where $\tilde u$ may be arbitrary smooth extension of $u$.
This differs from \cite{Ver87} and \cite{BaDe99,DiUr14} in which the errors were defined as $\|u - u_h\|_{H^1(\Omega\cap \Omega_h)}$ and $\|u - u_h\circ G_h\|_{H^1(\Omega)}$, respectively.
In view of practical computation, our choice of the error fits what is usually done in the numerical verification of convergence when $\Omega_h\neq\Omega$.
Compared with the method of Knobloch \cite{Kno99} who also employed $\|\tilde u - u_h\|_{H^1(\Omega_h)}$ as the error, the difference lies in the way the boundary element face $S$ on $\Gamma_h$ is mapped to a part of $\Gamma$.
In fact, he exploited the orthogonal projection from $\Gamma$ to $\Gamma_h$, the image of which is localized to each $S$.
Then he needed delicate arguments (see \cite[pp.\ 142--143]{Kno99}) to take into account the fact that it is not globally injective when $N=3$ (this point seems to be overlooked in \cite[p.\ 709]{Ver87}).
We, to the contrary, rely on the orthogonal projection $\pi$ from $\Gamma_h$ to $\Gamma$, which is globally bijective regardless of the space dimension $N$, provided the mesh size $h$ is sufficiently small.
This enables us to transfer the triangulation of $\Gamma_h$ to that on $\Gamma$ in a natural way, which is convenient to estimate surface integrals.
Complete proofs of the facts regarding $\pi$ used in this paper, which we could not find in the literature, are provided in Appendices \ref{sec:transformation between Gamma and Gammah} and \ref{sec:error of n and nh}.

Finally, we comment on the rate of convergence $O(h^{1/2})$ we obtain in our main result (\tref{thm:error estimate}) which is not optimal.
In our opinion, all the error estimates reported in the preceding works, which use $n_h$ to approximate $n$, remain $O(h^{1/2})$.
Verf\"urth \cite[Theorem 5.1]{Ver87} claimed $O(h)$; however, the estimate 
\begin{equation*}
	\left| \int_{\Gamma_h} u\cdot (n_h - n\circ\pi^{-1})\sigma_h \right| \le Ch \|u\|_{H^{1/2}(\Gamma_h)} \|\sigma_h\|_{H^{-1/2}(\Gamma_h)},
\end{equation*}
which was used to derive equation (5.12) there, seems non-trivial because $n_h$ is not smooth enough globally on $\Gamma_h$ (e.g.\ it does not belong to $H^{1/2}(\Gamma_h)$).
If $\|\sigma_h\|_{H^{-1/2}(\Gamma_h)}$ on the right-hand side is replaced by $\|\sigma_h\|_{L^2(\Gamma_h)}$, then one ends up with $O(h^{1/2})$ in the final estimate.
Dione and Urquiza \cite[Theorem 4]{DiUr14} claimed $O(h^{2/3})$; however, in equation (4.13) there, they did not consider the contribution
\begin{equation*}
	\frac1{\sqrt\epsilon} \|(u_\epsilon - \bar v_h)\cdot n\|_{L^2(\Gamma)},
\end{equation*}
which should appear inside the infimum over $v_h\in K_h$ even when $\Omega_h = \Omega$ (see Proposition 4.2 of Layton \cite{Lay99}).
If this contribution is taken into account, one obtains $O(h^{1/2})$ for the final result.

To overcome the sub-optimality, in \sref{sec:reduced-order integration} we investigate the penalty scheme in which reduced-order numerical integration is applied to the penalty term.
This method was proposed in \cite{DTU13} and was shown to be efficient by numerical experiments for $N=2$.
We give a rigorous justification for this observation in the sense that the error estimate improves to $O(h)$ if $\epsilon = O(h^2)$.
Our numerical example shows that the reduced-order numerical integration gives better results also for $N=3$, although this is not proved rigorously.

In our numerical results presented in \sref{sec:numerical examples}, we not only provide numerical verification of convergence but also discuss how the penalty parameter $\epsilon$ affects the performance of linear solvers.
We find that too small $\epsilon$ can lead to non-convergence of iterative methods such as GMRES, whereas sparse direct solvers such as UMFPACK always manage to solve the linear system.

\section{Formulation of the Stokes problem with slip boundary condition}
We present our notation for the function spaces and bilinear forms that we employ in this paper.
The boundary $\Gamma$ of $\Omega$ is supposed to be at least $C^{1,1}$-smooth.
The standard Lebesgue and Sobolev(-Slobodetski\u{\i}) spaces are denoted by $L^p(\Omega)$ and $W^{s,p}(\Omega)$ respectively, for $p\in[1, \infty]$ and $s\ge0$.
When $p=2$, we let $H^s(\Omega) := W^{s,2}(\Omega)$.
Their vectorial versions are indicated as $L^p(\Omega)^N$ and $W^{m, p}(\Omega)^N$ etc.; however, in case they appear as subscripts, say $\|\cdot\|_{L^p(\Omega)^N}$, we write $\|\cdot\|_{L^p(\Omega)}$ for the sake of simplicity.
The spaces above may also be defined for $\Gamma$; see e.g.\ \cite{Nec12}.

We define function spaces to describe velocity and pressure as follows:
\begin{align*}
	V = H^1(\Omega)^N, \qquad Q = L^2(\Omega).
\end{align*}
Also we set 
\begin{align*}
	V_n &:= \{ v\in H^1(\Omega)^N \,:\, (\mathrm{Tr}\,v)\cdot n = 0 \mbox{ on } \Gamma \}, \\
	\mathring Q &:= L^2_0(\Omega) = \{ q\in L^2(\Omega) \,:\, \textstyle\int_\Omega q\,dx = 0 \},
\end{align*}
where $\mathrm{Tr}$ stands for the trace operator; for simplicity, $\mathrm{Tr}\,v$ is indicated as $v$ in the following.
Finally, to describe a quantity which corresponds to the normal component of the traction vector, i.e.\ $\sigma(u, p)n\cdot n$, we introduce
\begin{equation} \label{eq:Lambda}
	\Lambda := H^{-1/2}(\Gamma) = H^{1/2}(\Gamma)',
\end{equation}
where the prime means the dual space.

Let $G\subset\mathbb R^N$ be an open set.
The $L^2(G)$- and $L^2(\partial G)$-inner products are denoted by $(\cdot, \cdot)_G$ and $(\cdot, \cdot)_{\partial G}$, respectively.
Moreover, we define bilinear forms $a_G,\, b_G,\, c_{\partial G}$, for $u,v\in H^1(G)^N,\, q\in L^2(G),\, \lambda,\mu\in L^2(\partial G)$, by
\begin{align}
	a_G(u, v) &= \int_G u\cdot v\,dx + \frac\nu2 \int_G (\nabla u + (\nabla u)^T) : (\nabla v + (\nabla v)^T) \,dx, \label{eq:bilinear form aG} \\
	b_G(v, q) &= -\int_G \mathrm{div}\,v\, q\,dx, \label{eq:bilinear form bG} \\
	c_{\partial G}(\lambda, \mu) &= \int_{\partial G} \lambda\mu\, d\gamma, \label{eq:bilinear form cG}
\end{align}
where the dot and colon mean the inner products for vectors and matrices, respectively.
When $G = \Omega$, we use the abbreviation $a=a_\Omega,\, b=b_\Omega,\, c=c_{\partial\Omega}$.
In the following, $c(\cdot, \cdot)$ is also interpreted as the duality pairing between $H^{1/2}(\Gamma)$ and $H^{-1/2}(\Gamma)$.

The variational formulation for \eref{eq:Stokes slip BC} consists in finding $(u, p)\in V\times\mathring Q$ such that $u\cdot n = g$ on $\Gamma$ and
\begin{equation} \label{eq:weak form with constraint}
	\left\{
	\begin{aligned}
		a(u, v) + b(v, p) &= (f, v)_\Omega + (\tau, v)_\Gamma \quad && \forall v\in V_n, \\
		b(u, q) &= 0 && \forall q\in\mathring Q.
	\end{aligned}
	\right.
\end{equation}
It is well known that this variational equation admits a unique solution under the compatibility condition \eref{eq:compatibility condition for g} and that it becomes regular according to the smoothness of $\Gamma,\, f,\, g,\, \tau$; see \cite{BdV04,SoSc73}.
Throughout this paper, we assume $\Gamma\in C^{2,1},\, f\in L^2(\Omega)^N,\, g\in H^{3/2}(\Gamma),\, \tau\in H^{1/2}(\Gamma)^N$, so that the regularity $(u, p) \in H^2(\Omega)^N \times H^1(\Omega)$ is assured.

Letting $\lambda := -\sigma(u, p)n\cdot n$, we see that $(u, p, \lambda) \in V\times Q\times \Lambda$ satisfies
\begin{equation} \label{eq:weak form without constraint}
	\left\{
	\begin{aligned}
		a(u, v) + b(v, p) + c(v\cdot n, \lambda) &= (f, v)_\Omega + (\tau, v)_\Gamma \quad && \forall v\in V, \\
		b(u, q) &= 0 && \forall q\in Q, \\
		c(u\cdot n - g, \mu) &= 0 && \forall \mu\in\Lambda.
	\end{aligned}
	\right.
\end{equation}
In fact, \eref{eq:weak form without constraint}$_1$ follows from Green's formula.
By \eref{eq:compatibility condition for g} one has $b(u, 1) = 0$, which implies \eref{eq:weak form without constraint}$_2$.
Multiplying $u\cdot n = g$ by a test function $\mu$ and integrating over $\Gamma$ lead to \eref{eq:weak form without constraint}$_3$.

\begin{rem} \label{rem:additive constant}
	Observe that $(u, p+k, \lambda+k)$ with any $k\in\mathbb R$ is also a solution of \eref{eq:weak form without constraint}.
	According to this fact, we will adjust the additive constant of $p$ (and thus of $\lambda$) later on, before we start error analysis of the finite element approximation (see \rref{rem:adjustment of additive constant} below).
\end{rem}

\section{Finite element approximation} \label{sec:finite element approximation}
\subsection{Triangulation and FE spaces} \label{subsec:triangulation}
Let us introduce a regular family of triangulations $\mathcal T_h$ of a polyhedral domain $\Omega_h$, which is assigned the \emph{mesh size} $h>0$.
Namely, we assume that:
\begin{enumerate}[(H1)]
	\item each $T\in\mathcal T_h$ is a closed $N$-simplex such that $h_T := \mathrm{diam}\,T \le h$;
	\item $\Omega_h = \bigcup_{T\in\mathcal T_h}T$;
	\item the intersection of any two distinct elements is empty or consists of their common face of dimension $\le N-1$;
	\item there exists a constant $C>0$, independent of $h$, such that $\rho_T \le Ch_T$ for all $T\in\mathcal T_h$ where $\rho_T$ denotes the diameter of the inscribed ball of $T$.
\end{enumerate}
We define the \emph{boundary mesh} $\mathcal S_h$ inherited from $\mathcal T_h$ by
\begin{equation*}
	\mathcal S_h = \{ S\in\mathcal S_h \,:\, \text{$S$ is an $(N-1)$-face of some $T\in\mathcal T_h$} \}.
\end{equation*}
Then we see that $\mathcal S_h$ satisfies the requirements that are analogous to (H1)--(H4) above, and especially we have $\Gamma_h := \partial\Omega_h = \bigcup_{S\in\mathcal S_h} S$.
We assume that $\Gamma_h$ approximates $\Gamma$ in the following sense:
\begin{enumerate}[(H1)]
	\setcounter{enumi}{4}
	\item the vertices of every $S\in\mathcal S_h$ lie on $\Gamma$.
\end{enumerate}
Throughout this paper, we confine ourselves to the case where $h>0$ is sufficiently small, which will not be emphasized in the following.
In particular, all the results given in Appendices \ref{sec:transformation between Gamma and Gammah} and \ref{sec:error of n and nh} are supposed to hold true.

As mentioned in \sref{sec:intro}, we focus on the P1/P1 and P1b/P1 finite element approximations for velocity and pressure, to which we refer as $l=1$ and $l=1b$, respectively.
Namely, we define
\begin{align*}
	V_h &=
	\begin{cases}
		\{ v_h \in C(\overline\Omega_h)^N \,:\, v_h|_T \in P_1(T)^N \quad \forall T\in\mathcal T_h \} & \text{if} \quad l=1, \\
		\{ v_h \in C(\overline\Omega_h)^N \,:\, v_h|_T \in P_1(T)^N \oplus B(T)^N \quad \forall T\in\mathcal T_h \} & \text{if} \quad l=1b,
	\end{cases}
	\\
	Q_h &= \{ q_h \in C(\overline\Omega_h) \,:\, q_h|_T \in P_1(T) \quad \forall T\in\mathcal T_h \},
\end{align*}
where $B(T)$ stands for the space spanned by the bubble function on $T$.
We also set $\mathring V_h := V_h \cap H^1_0(\Omega_h)^N$ and $\mathring Q_h := Q_h \cap L^2_0(\Omega_h)$.
We consider a discrete version of the space $\Lambda$ (recall \eref{eq:Lambda}) as
\begin{equation*}
	\Lambda_h = \{ \lambda_h \in L^2(\Gamma_h) \,:\, \lambda_h|_S \in P_1(S) \quad \forall S\in\mathcal S_h \},
\end{equation*}
which is a discontinuous $P_1$ finite element space on $\Gamma_h$.

We turn our attention to interpolation operators.
In order to deal with the situation $\Omega\neq\Omega_h$, we first extend functions defined in $\Omega$ to $\tilde\Omega$, which is a fixed bounded smooth domain containing $\overline{\Omega\cup\Omega_h}$.
For this purpose we consider an arbitrary extension operator $P: W^{m, p}(\Omega)\to W^{m, p}(\tilde\Omega)$ satisfying the stability condition
\begin{equation*}
	\|Pf\|_{W^{m,p}(\tilde\Omega)} \le C\|f\|_{W^{m,p}(\Omega)} \qquad \forall f\in W^{m,p}(\Omega),
\end{equation*}
where $C$ is a constant depending only on $N,\, \Omega,\, m,\, p$.
Such $P$ does exist; for example, if $m=0$ we may exploit the zero-extension as $P$, and if $m\ge1$ it can be constructed e.g.\ by Nikolskii's method (see \cite{Nec12}).
In view of the lift theorem which concerns a right inverse of the trace operator, we may also extend functions given on $\Gamma$ to those in $\tilde \Omega$.
We agree to use the same symbol $P$ to refer to this extension as well.
Then, given a function $f$ in $\Omega$ or on $\Gamma$, we denote $Pf$ by $\tilde f$ for simplicity in the notation.

Now we let $I_h$ and $R_h$ represent the Lagrange interpolation operator to linear FE spaces and a local regularization operator to linear FE spaces, respectively.
Then, from the theory of interpolation error estimates (see \cite[Section 4]{BrSc07}) combined with the stability of extension operators, we have
\begin{align*}
	\|\tilde f - I_h\tilde f\|_{H^m(\Omega_h)}   &\le Ch^{2-m} \|f\|_{H^2(\Omega)}, &&f\in H^2(\Omega),\, m=0,1, \\
	\|\tilde f - R_h\tilde f\|_{H^m(\Omega_h)} &\le Ch^{1-m} \|f\|_{H^1(\Omega)}, &&f\in H^1(\Omega),\, m=0,1, 
\end{align*}
where the constants $C$ depend only on $N$, on the constant in assumption (H4) above, and on a reference element.
For $L^2(\Gamma_h)$-estimates, we first apply a trace inequality on each boundary element and then add them up to obtain
\begin{align*}
	\|\tilde f - I_h\tilde f\|_{L^2(\Gamma_h)}  &\le C\|\tilde f - I_h\tilde f\|_{L^2(\Omega_h)}^{1/2} \|\tilde f - I_h\tilde f\|_{H^1(\Omega_h)}^{1/2} \\
								     &\le Ch^{3/2} \|\tilde f\|_{H^2(\tilde\Omega)} \le Ch^{3/2}\|f\|_{H^{3/2}(\Gamma)} &&f\in H^{3/2}(\Gamma), \\
	\|\tilde f - R_h\tilde f\|_{L^2(\Gamma_h)} &\le Ch^{1/2} \|f\|_{H^{1/2}(\Gamma)} &&f\in H^{1/2}(\Gamma).
\end{align*}

\subsection{FE scheme with penalty}
We propose the following discrete problem to approximate \eref{eq:Stokes slip BC}:
choose $\epsilon>0$ suitably small and find $(u_h, p_h) \in V_h\times Q_h$ such that, for all $(v_h, q_h) \in V_h\times Q_h$,
\begin{equation} \label{eq:discrete problem with 2 variables}
	\left\{
	\begin{aligned}
		a_h(u_h, v_h) + b_h(v_h, p_h) + \frac1\epsilon c_h(u_h\cdot n_h - I_h\tilde g, v_h\cdot n_h) &= (\tilde f, v_h)_{\Omega_h} + (\tilde\tau, v_h)_{\Gamma_h}, \\
		b_h(u_h, q_h) &= d_h(p_h, q_h),
	\end{aligned}
	\right.
\end{equation}
where $n_h$ is the outer unit normal of $\Gamma_h$, and $a_h := a_{\Omega_h},\, b_h := b_{\Omega_h},\, c_h := c_{\Gamma_h}$ (recall \eref{eq:bilinear form aG}--\eref{eq:bilinear form cG} above).
Moreover, $\tilde f,\, \tilde g,\, \tilde\tau$ represent the extensions of $f,\, g,\, \tau$ respectively, which are discussed in the previous subsection.
$d_h(\cdot, \cdot)$ is a pressure-stabilizing term, which is present only when $l=1$ and is defined by
\begin{equation*}
	d_h(p_h, q_h) = \eta h^2 (\nabla p_h, \nabla q_h)_{\Omega_h}, \qquad \eta := \left\{
	\begin{aligned}
		1 && \mbox{ if } & \quad l=1, \\
		0 && \mbox{ if } & \quad l=1b.
	\end{aligned} \right.
\end{equation*}
We remark that $\eta$ for the case $l=1$ can be any positive constant; here, we suppose it to be $1$ for simplicity.
We state the well-posedness of this discrete problem.

\begin{prop} \label{prop:discrete well-posedness}
	There exists a unique solution $(u_h, p_h) \in V_h\times Q_h$ of \eref{eq:discrete problem with 2 variables}.
\end{prop}
The proof relies on the following discrete versions of Korn's inequality and the inf-sup condition:
\begin{align}
	\alpha \|v_h\|_{H^1(\Omega_h)}^2 &\le a_h(v_h, v_h) && \forall v_h\in V_h, \label{eq:uniform Korn inequality} \\
	C \|q_h\|_{L^2(\Omega_h)} &\le \sup_{v_h\in\mathring V_h} \frac{b(v_h, q_h)}{\|v_h\|_{H^1(\Omega_h)}} +  C\eta h\|\nabla q_h\|_{L^2(\Omega_h)} && \forall q_h\in\mathring Q_h, \label{eq:uniform inf-sup condition}
\end{align}
where $\alpha>0$ depends only on $N,\Omega,\nu$ and $C>0$ depends only on $N,\Omega$.
Proofs of these uniform coercivity estimates are found in \cite[Theorem 4.3]{Kno99} and in \cite[Proposition 4]{Tab01}, respectively.
They will be of central importance when we perform error analysis in \sref{sec:error analysis}.
Now we prove the proposition.

\begin{proof}[Proof of \pref{prop:discrete well-posedness}]
	We notice that \eref{eq:discrete problem with 2 variables} is equivalently rewritten as follows: find $(u_h, p_h)\in V_h\times Q_h$ such that, for all $(v_h, q_h)\in V_h\times Q_h$,
	\begin{align}
		B_h(u_h, p_h; v_h, q_h) := &\, a_h(u_h, v_h) + b_h(v_h, p_h) - b_h(u_h, q_h) + d_h(p_h, q_h) + \frac1\epsilon c_h(u_h\cdot n_h, v_h\cdot n_h) \notag \\
						       = &\, (\tilde f, v_h)_{\Omega_h} + (\tilde\tau, v_h)_{\Gamma_h} + \frac1\epsilon c_h(I_h\tilde g, v_h\cdot n_h). \label{eq:Bh}
	\end{align}
	Because the problem is finite dimensional, it suffices to prove that $B_h(u_h, p_h; v_h, q_h)=0$ for all $v_h$ and $q_h$ implies $u_h=p_h=0$.
	Taking $(v_h, q_h) = (u_h, p_h)$ yields, thanks to \eref{eq:uniform Korn inequality}, $u_h=0$ and $d_h(p_h, p_h) = 0$.
	Below we deal with $l=1$ and $l=1b$ separately.
	
	Let $l=1$. Then, since $\eta>0$ we have $\nabla p_h = 0$, which implies that $p_h$ equals some constant $k$.
	By \eref{eq:Bh} one has $b_h(v_h, k) = -k \int_{\Gamma_h} v_h\cdot n_h\,d\gamma_h = 0$ for all $v_h\in V_h$.
	Choosing $v_h\in V_h$ such that $\int_{\Gamma_h} v_h\cdot n_h\,d\gamma_h \neq 0$ gives $k = 0$, so that $p_h = 0$.
	
	When $l=1b$, \eref{eq:Bh} reduces to $b(v_h, q_h) = 0$ for all $v_h\in V_h$.
	This combined with \eref{eq:uniform inf-sup condition} tells us that $p_h$ is equal to a constant.
	Discussing as in $l=1$, we conclude $p_h = 0$.
	This completes the proof of \pref{prop:discrete well-posedness}.
\end{proof}

Let us rewrite \eref{eq:discrete problem with 2 variables} in such a way that the discrete problem becomes comparable with the continuous one given in \eref{eq:weak form without constraint}.
To this end, we introduce $\lambda_h := \frac1\epsilon (u_h\cdot n_h - I_h\tilde g) \in \Lambda_h$ to obtain
\begin{equation} \label{eq:discrete problem with 3 variables}
	\left\{
	\begin{aligned}
		a_h(u_h, v_h) + b_h(v_h, p_h) +c_h(v_h\cdot n_h, \lambda_h) &= (\tilde f, v_h)_{\Omega_h} + (\tilde\tau, v_h)_{\Gamma_h} & \forall v_h\in V_h, \\
		b_h(u_h, q_h) &= d_h(p_h, q_h) & \forall q_h\in Q_h, \\
		c_h(u_h\cdot n_h - I_h\tilde g, \mu_h) &= \epsilon c_h(\lambda_h, \mu_h) & \forall \mu_h\in\Lambda_h.
	\end{aligned}
	\right.
\end{equation}
In fact, \eref{eq:discrete problem with 3 variables}$_1$ results from the symmetry $c_h(\lambda_h, \mu_h) = c_h(\mu_h, \lambda_h)$, and \eref{eq:discrete problem with 3 variables}$_3$ follows from multiplying the equation $u_h\cdot n_h - I_h\tilde g = \epsilon \lambda_h$ by any test function in $L^2(\Gamma_h) \supset \Lambda_h$.

\subsection{Auxiliary lemmas}
We collect several results which will be useful to evaluate the difference of $\Omega$ and $\Omega_h$ in terms of volume or surface integrals.
We denote the symmetric difference of $\Omega$ and $\Omega_h$ by
\begin{equation*}
	\Omega\triangle\Omega_h := (\Omega\setminus\Omega_h)\cup(\Omega_h\setminus\Omega)
\end{equation*}
and call it the \emph{boundary skin}.
The first lemma concerns a linear operator which extends functions in $V_h$ to $\tilde\Omega$ is equipped with suitable stability properties.
For the proof, we refer to \cite[Theorem 4.1]{Kno99}.
\begin{lem} \label{lem:discrete extension}
	There exists a linear operator $P_h: V_h\to H^1(\tilde\Omega)^N$ such that $P_hv_h|_{\overline\Omega_h} = v_h$ and
	\begin{align*}
		\|P_hv_h\|_{H^m(\tilde\Omega)} &\le C\|v_h\|_{H^m(\Omega_h)} && \forall v_h \in V_h,\, m=0,1, \\
		\|P_hv_h\|_{H^m(\Omega\triangle\Omega_h)} &\le Ch^{1/2}\|v_h\|_{H^m(\Omega_h)} && \forall v_h \in V_h,\, m=0,1,
	\end{align*}
	where the constants $C$ are independent of $h$.
\end{lem}

The second lemma gives estimates of volume integrals over the boundary skin, which are not restricted to discrete spaces.
However, notice that we require higher regularity for the right-hand side.
\begin{lem} \label{lem:boundary skin estimate}
	Let $f\in H^{m+1}(\tilde\Omega),\, m=0,1$. Then we have
	\begin{align*}
		\|f\|_{H^m(\Omega\triangle\Omega_h)} &\le Ch \|f\|_{H^{m+1}(\tilde\Omega)},
	\end{align*}
	where $C$ is independent of $h$.
\end{lem}
\begin{proof}
	Since $\Omega\triangle\Omega_h \subset \Gamma(\tilde C_{0E}h^2)$ by \pref{prop:estimate of Phi}, \tref{thm:Lp norm in tubular neighborhood} for $\delta_1 = \tilde C_{0E}h^2$ and for $p=2$, combined with the trace inequality $\|f\|_{L^2(\Gamma)} \le C\|f\|_{H^1(\Omega)}$, leads to the desired result.
\end{proof}

The last lemma restates results concerning surface integrals obtained in Theorems \ref{thm:error estimate of surface integral} and \ref{thm:trace and transform}.
\begin{lem} \label{lem:surface integral estimate}
	Let $\pi$ be the orthogonal projection to $\Gamma$ defined in a tubular neighborhood of $\Gamma$. Then we have the following estimates:
	\begin{align*}
		\|f\circ\pi\|_{L^2(\Gamma_h)} &\le C\|f\|_{L^2(\Gamma)} && \forall f\in L^2(\Gamma), \\
		\left| \int_{\Gamma} f\,d\gamma - \int_{\Gamma_h} f\circ\pi\,d\gamma_h \right| &\le Ch^2 \|f\|_{L^1(\Gamma)} && \forall f\in L^1(\Gamma), \\
		\|f - f\circ\pi\|_{L^2(\Gamma_h)} &\le Ch \|f\|_{H^1(\tilde\Omega)} && \forall f\in H^1(\tilde\Omega),
	\end{align*}
	where the constants $C$ are independent of $h$.
	Here, $d\gamma$ and $d\gamma_h$ denote the surface elements associated with $\Gamma$ and $\Gamma_h$, respectively.
\end{lem}
\begin{rem} \label{rem:L2 norm on Gammah}
	By the first and third estimates together with trace inequalities, for all $f\in H^{1/2}(\Gamma)$ (hence its extension $\tilde f$ is in $H^1(\tilde\Omega)$) we obtain
	\begin{equation*}
		\|\tilde f\|_{L^2(\Gamma_h)} \le \|\tilde f - f\circ\pi\|_{L^2(\Gamma_h)} + \|f\circ\pi\|_{L^2(\Gamma_h)} \le C\|\tilde f\|_{H^1(\tilde\Omega)} \le C\|f\|_{H^{1/2}(\Gamma)}.
	\end{equation*}
\end{rem}

\section{Error analysis of penalty FE scheme} \label{sec:error analysis}
\subsection{Estimation of consistency error}
To explain the idea, suppose that $\Omega = \Omega_h$ and that $g=0$. Then subtracting the discrete problem \eref{eq:discrete problem with 3 variables} from the continuous one \eref{eq:weak form without constraint} would give us
\begin{equation} \label{eq:Galerkin orthogonality}
	\left\{
	\begin{aligned}
		a(u-u_h, v_h) + b(v_h, p-p_h) + c(v_h\cdot n, \lambda-\lambda_h) &= 0 & \forall v_h\in V_h, \\
		b(u-u_h, q_h) &= -d_h(p_h, q_h) & \forall q_h\in Q_h, \\
		c((u-u_h)\cdot n, \mu_h) &= -\epsilon c(\lambda_h, \mu_h) & \forall \mu_h\in\Lambda_h.
	\end{aligned}
	\right.
\end{equation}
This type of relation is essentially important in error analysis of the finite element method and is sometimes called the ``Galerkin orthogonality''.
If $\Omega\neq\Omega_h$, then such a relation is by no means available because subtraction is impossible.
However, even in this case one can still expect that an asymptotic version of \eref{eq:Galerkin orthogonality} should hold as $h$ becomes small.
The next proposition verifies this expectation.

\begin{prop} \label{prop:consistency error}
	Let $(u, p, \lambda)$ and $(u_h, p_h, \lambda_h)$ be solutions of \eref{eq:weak form without constraint} and \eref{eq:discrete problem with 3 variables} respectively.
	We assume the regularity of the data: $\Gamma\in C^{2,1},\, f\in L^2(\Omega)^N,\, g\in H^{3/2}(\Gamma),\, \tau\in H^{1/2}(\Gamma)^N$.
	Then there exist constants $C=C(N, \Omega, \nu, f, g, \tau)$, independent of $h$ and $\epsilon$, such that for all $(v_h, q_h, \mu_h) \in V_h\times Q_h\times\Lambda_h$ there holds
	\begin{equation} \label{eq:consistency error}
		\left\{
		\begin{aligned}
			|a_h(\tilde u - u_h, v_h) + b_h(v_h, \tilde p - p_h) + c_h(v_h\cdot n, \tilde \lambda - \lambda_h)| &\le Ch\|v_h\|_{H^1(\Omega_h)}, \\
			|b_h(\tilde u - u_h, q_h) + d_h(p_h, q_h)| &\le Ch\|q_h\|_{L^2(\Omega_h)}, \\
			|c_h((\tilde u - u_h)\cdot n_h, \mu_h) + \epsilon c_h(\lambda_h, \mu_h)| &\le Ch\|\mu_h\|_{L^2(\Gamma_h)}.
		\end{aligned}
		\right.
	\end{equation}
\end{prop}
\begin{proof}
	(i) Let us prove \eref{eq:consistency error}$_1$. Since $\int_{\Omega_h} = \int_\Omega + \int_{\Omega_h\setminus\Omega} - \int_{\Omega\setminus\Omega_h}$, we obtain
	\begin{align*}
		a_h(\tilde u - u_h, v_h) &= a(u, P_hv_h) - a_h(u_h, v_h) + a_{\Omega_h\setminus\Omega}(\tilde u, v_h) - a_{\Omega\setminus\Omega_h}(u, P_hv_h), \\
		b_h(v_h, \tilde p - p_h) &= b(P_hv_h, \tilde p) - b_h(v_h, p_h) + b_{\Omega_h\setminus\Omega}(v_h, \tilde p) - b_{\Omega\setminus\Omega_h}(P_hv_h, p). \\
		\intertext{On the other hand, it is clear that}
		c_h(v_h\cdot n_h, \tilde\lambda - \lambda_h) &= c(P_hv_h\cdot n, \tilde\lambda) - c_h(v_h\cdot n_h, \lambda_h) + c_h(v_h\cdot n_h, \tilde\lambda) - c(P_hv_h\cdot n, \lambda).
	\end{align*}
	Addition of the three equations above combined with \eref{eq:weak form without constraint}$_1$ and \eref{eq:discrete problem with 3 variables}$_1$ yields
	\begin{align}
		     &a_h(\tilde u - u_h, v_h) + b_h(v_h, \tilde p - p_h) + c_h(v_h\cdot n, \tilde \lambda - \lambda_h) \notag \\
		=\; & (f, P_hv_h)_\Omega - (\tilde f, v_h)_{\Omega_h} + (\tau, P_hv_h)_\Gamma - (\tilde\tau, v_h)_{\Gamma_h} \notag \\
		     &\quad + a_{\Omega_h\setminus\Omega}(\tilde u, v_h) - a_{\Omega\setminus\Omega_h}(u, P_hv_h) + b_{\Omega_h\setminus\Omega}(v_h, \tilde p) - b_{\Omega\setminus\Omega_h}(P_hv_h, p) \notag \\
		     &\quad + c_h(v_h\cdot n_h, \tilde\lambda) - c(P_hv_h\cdot n, \lambda) \notag \\
		=:  &\, I_1 + I_2 + I_3 + I_4 + I_5. \label{eq:representation of consistency error}
	\end{align}
	
	First we estimate volume integrals.
	It follows from Lemmas \ref{lem:discrete extension} and \ref{lem:boundary skin estimate} together with the stability of the extensions that
	\begin{align*}
		|I_1| &\le \|\tilde f\|_{L^2(\Omega\triangle\Omega_h)} \|P_hv_h\|_{L^2(\Omega\triangle\Omega_h)} \le Ch \|f\|_{L^2(\Omega)} \|v_h\|_{H^1(\Omega_h)}, \\
		|I_3| &\le C \|\tilde u\|_{H^1(\Omega\triangle\Omega_h)} \|P_hv_h\|_{H^1(\Omega\triangle\Omega_h)} \le Ch^{7/6} \|u\|_{H^2(\Omega)} \|v_h\|_{H^1(\Omega_h)}, \\
		|I_4| &\le C \|P_hv_h\|_{H^1(\Omega\triangle\Omega_h)} \|\tilde p\|_{L^2(\Omega\triangle\Omega_h)} \le Ch^{7/6} \|p\|_{H^1(\Omega)} \|v_h\|_{H^1(\Omega_h)}.
	\end{align*}
	
	Next we estimate surface integrals. 
	For $I_2$ we observe that
	\begin{align*}
		I_2 &= \int_\Gamma \tau\cdot P_hv_h\,d\gamma - \int_{\Gamma_h} (\tau\cdot P_hv_h)\circ\pi \,d\gamma_h + (\tau\circ\pi - \tilde\tau, (P_hv_h)\circ\pi)_{\Gamma_h} \\
		      &\hspace{6.2cm} + (\tilde\tau, (P_hv_h)\circ\pi - P_hv_h)_{\Gamma_h} \\
		      &=: I_{21} + I_{22} + I_{23}.
	\end{align*}
	It follows from \lref{lem:surface integral estimate}, \rref{rem:L2 norm on Gammah}, and \lref{lem:discrete extension} that
	\begin{align*}
		|I_{21}| &\le Ch^2 \|\tau\cdot P_hv_h\|_{L^1(\Gamma)} \le Ch^2 \|\tau\|_{L^2(\Gamma)} \|v_h\|_{H^1(\Omega_h)}, \\
		|I_{22}| &\le Ch \|\tilde\tau\|_{H^1(\tilde\Omega)} \|P_hv_h\|_{H^1(\tilde\Omega)} \le Ch\|\tau\|_{H^{1/2}(\Gamma)} \|v_h\|_{H^1(\Omega_h)}, \\
		|I_{23}| &\le \|\tilde\tau\|_{L^2(\Gamma_h)}\cdot Ch\|P_hv_h\|_{H^1(\tilde\Omega)} \le Ch \|\tau\|_{H^{1/2}(\Gamma)} \|v_h\|_{H^1(\Omega_h)}.
	\end{align*}
	For $I_5$ we observe that
	\begin{align*}
		I_5 &= c_h(v_h\cdot (n_h - n\circ\pi), \tilde\lambda) + c_h( (v_h - P_hv_h)\cdot (n\circ\pi), \tilde\lambda) \\
		      &\hspace{3.9cm} + c_h((P_hv_h\cdot n)\circ\pi, \tilde\lambda - \lambda\circ\pi) \\
		      &\qquad + \int_{\Gamma_h} (P_hv_h\cdot n\,\lambda)\circ\pi\,d\gamma_h - \int_{\Gamma} P_hv_h\cdot n\,\lambda\,d\gamma
		      =: I_{51} + I_{52} + I_{53} + I_{54}.
	\end{align*}
	It follows from Lemmas \ref{lem:n and nh}, \ref{lem:surface integral estimate}, and \ref{lem:discrete extension} that
	\begin{align*}
		|I_{51}| &\le Ch\|v_h\|_{L^2(\Gamma_h)} \|\tilde \lambda\|_{L^2(\Gamma_h)} \le Ch(\|u\|_{H^2(\Omega)} + \|p\|_{H^1(\Omega)}) \|v_h\|_{H^1(\Omega_h)}, \\
		|I_{52}| &\le \|P_hv_h - (P_hv_h)\circ\pi\|_{L^2(\Gamma_h)} \|\tilde\lambda\|_{L^2(\Gamma_h)} \le Ch \|P_hv_h\|_{H^1(\tilde\Omega)} \|\tilde\lambda\|_{L^2(\Gamma_h)} \\
			    &\le Ch (\|u\|_{H^2(\Omega)} + \|p\|_{H^1(\Omega)}) \|v_h\|_{H^1(\Omega_h)}, \\
		|I_{53}| &\le \|(P_hv_h\cdot n)\circ\pi\|_{L^2(\Gamma_h)} \|\tilde\lambda - \lambda\circ\pi\|_{L^2(\Gamma_h)} \le C\|P_hv_h\|_{L^2(\Gamma)} \cdot Ch\|\tilde\lambda\|_{H^1(\tilde\Omega)} \\
			    &\le Ch (\|u\|_{H^2(\Omega)} + \|p\|_{H^1(\Omega)}) \|v_h\|_{H^1(\Omega_h)}, \\
		|I_{54}| &\le Ch^2 \|P_hv_h\cdot n\,\lambda\|_{L^1(\Gamma)} \le Ch^2 \|P_hv_h\|_{L^2(\Gamma)} \|\lambda\|_{L^2(\Gamma)} \\
			    &\le Ch^2 (\|u\|_{H^2(\Omega)} + \|p\|_{H^1(\Omega)}) \|v_h\|_{H^1(\Omega_h)}.
	\end{align*}
	Combining the above estimates with \eref{eq:representation of consistency error} and noting that $\|u\|_{H^2(\Omega)} + \|p\|_{H^1(\Omega)} \le C(\|f\|_{L^2(\Omega)} + \|g\|_{H^{3/2}(\Gamma)} + \|\tau\|_{H^{1/2}(\Gamma)})$ by the regularity theory of the Stokes equations, we conclude \eref{eq:consistency error}$_1$.
	
	(ii) Let us prove \eref{eq:consistency error}$_2$.
	One finds from \eref{eq:weak form without constraint}$_2$ and \eref{eq:discrete problem with 3 variables}$_2$ that
	\begin{align*}
		b_h(\tilde u - u_h, q_h) + d_h(p_h, q_h) &= b_h(\tilde u, q_h) - b(u, P_hq_h) =: I_6.
	\end{align*}
	By Lemmas \ref{lem:discrete extension} and \ref{lem:boundary skin estimate} we have
	\begin{equation*}
		|I_6| \le C\|\tilde u\|_{H^1(\Omega\triangle\Omega_h)} \|P_hq_h\|_{L^2(\Omega\triangle\Omega_h)} \le Ch^{7/6} \|u\|_{H^2(\Omega)} \|q_h\|_{L^2(\Omega_h)},
	\end{equation*}
	from which \eref{eq:consistency error}$_2$ follows.
	
	(iii) Let us prove \eref{eq:consistency error}$_3$.
	One finds from \eref{eq:weak form without constraint}$_3$ and $u\cdot n = g$ on $\Gamma$, which is due to \eref{eq:discrete problem with 3 variables}$_3$, that
	\begin{align*}
		     &c_h((\tilde u - u_h)\cdot n_h, \mu_h) + \epsilon c_h(\lambda_h, \mu_h) = c_h(\tilde u\cdot n_h - I_h\tilde g, \mu_h) \\
		=\; &c_h(\tilde u\cdot(n_h - n\circ\pi), \mu_h) + c_h((\tilde u - u\circ\pi)\cdot n\circ\pi, \mu_h) + c_h(g\circ\pi - \tilde g, \mu_h) \\
		     &\hspace{7.7cm} + c_h(\tilde g - I_h\tilde g, \mu_h) \\
		=:&\,I_7 + I_8 + I_9 + I_{10}.
	\end{align*}
	By Lemmas \ref{lem:n and nh}, \ref{lem:surface integral estimate}, and \ref{lem:discrete extension}, together with the interpolation error estimate, we have
	\begin{align*}
		|I_7| &\le Ch \|\tilde u\|_{L^2(\Gamma_h)} \|\mu_h\|_{L^2(\Gamma_h)} \le Ch \|u\|_{H^1(\Omega)} \|\mu_h\|_{L^2(\Gamma_h)}, \\
		|I_8| &\le \|\tilde u - u\circ\pi\|_{L^2(\Gamma_h)} \|\mu_h\|_{L^2(\Gamma_h)} \le Ch\|u\|_{H^1(\Omega)} \|\mu_h\|_{L^2(\Gamma_h)}, \\
		|I_9| &\le Ch \|\tilde g\|_{H^1(\tilde\Omega)} \|\mu_h\|_{L^2(\Gamma_h)} \le Ch \|g\|_{H^{1/2}(\Gamma)} \|\mu_h\|_{L^2(\Gamma_h)}, \\
		|I_{10}| &\le Ch^{3/2} \|\tilde g\|_{H^2(\Omega_h)} \|\mu_h\|_{L^2(\Gamma_h)} \le Ch^{3/2} \|g\|_{H^{3/2}(\Gamma)} \|\mu_h\|_{L^2(\Gamma_h)},
	\end{align*}
	from which \eref{eq:consistency error}$_3$ follows. This completes the proof of \pref{prop:consistency error}.
\end{proof}

\subsection{Error estimate for velocity and pressure}
We are ready to state the main result of this paper.
\begin{thm} \label{thm:error estimate}
	Assume that $\Gamma\in C^{2,1},\, f\in L^2(\Omega)^N,\, g\in H^{3/2}(\Gamma),\, \tau\in H^{1/2}(\Gamma)^N$ and that $h>0$ is sufficiently small.
	We let $(u, p)$ and $(u_h, p_h)$ be solutions of \eref{eq:weak form with constraint} and \eref{eq:discrete problem with 2 variables} respectively.
	Then there exists a constant $C=C(N,\Omega,\nu,f,g,\tau)$, independent of $h$ and $\epsilon$, such that
	\begin{equation*} \textstyle
		\|\tilde u - u_h\|_{H^1(\Omega_h)} + \|(\tilde p+k_h) - p_h\|_{L^2(\Omega_h)} \le C(\sqrt h + \sqrt\epsilon + \frac{h}{\sqrt\epsilon}),
	\end{equation*}
	where $k_h = \frac1{\mathrm{meas}(\Omega_h)} (p_h - R_h\tilde p, 1)_{\Omega_h}$.
\end{thm}

\begin{rem} \label{rem:adjustment of additive constant}
	As mentioned in \rref{rem:additive constant}, there is room for us to choose arbitrary additive constant $k$ for $p$.
	Here, we adjust $k$ in such a way that $R_h(\tilde p+k) - p_h \in L^2_0(\Omega_h)$, which gives rise to the constant $k_h$ above.
	By considering $(p+k_h, \lambda + k_h)$ instead of $(p, \lambda)$, we may assume $k_h = 0$ in the following proof.
\end{rem}

\begin{proof}[Proof of \tref{thm:error estimate}]
	Let $v_h := I_h\tilde u,\, q_h := R_h\tilde p,\, \mu_h := R_h\tilde\lambda$.
	By interpolation error estimates one has
	\begin{align*}
		\|\tilde u - u_h\|_{H^1(\Omega_h)} &\le C\|u\|_{H^2(\Omega)}h + \|v_h - u_h\|_{H^1(\Omega_h)}, \\
		\|\tilde p - p_h\|_{L^2(\Omega_h)} &\le C\|p\|_{H^1(\Omega)}h + \|q_h - p_h\|_{L^2(\Omega_h)},
	\end{align*}
	hence it suffices to bound $\|v_h - u_h\|_{H^1(\Omega_h)}$ and $\|q_h - p_h\|_{H^1(\Omega_h)}$ by $C(h^{1/2} + \epsilon^{1/2} + h/\epsilon^{1/2})$.
	According to the uniform ellipticity \eref{eq:uniform Korn inequality} we obtain
	\begin{align}
		       &\alpha\|v_h - u_h\|_{H^1(\Omega_h)}^2 + d_h(p_h - q_h, p_h - q_h) + \epsilon c_h(\lambda_h - \mu_h, \lambda_h - \mu_h) \notag \\
		\le\; &a_h(v_h - u_h, v_h - u_h) + d_h(p_h - q_h, p_h - q_h) + \epsilon c_h(\lambda_h - \mu_h, \lambda_h - \mu_h) \notag \\
		=\;   &a_h(v_h - \tilde u, v_h - u_h) \notag \\
		       &+ a_h(\tilde u - u_h, v_h - u_h) + b_h(v_h - u_h, \tilde p - p_h) + c_h((v_h - u_h)\cdot n_h, \tilde\lambda - \lambda_h) \notag \\
		       &-  b_h(v_h - u_h, \tilde p - p_h) + d_h(p_h - q_h, p_h - q_h) \notag \\
		       &-  c_h((v_h - u_h)\cdot n_h, \tilde\lambda - \lambda_h) + \epsilon c_h(\lambda_h - \mu_h, \lambda_h - \mu_h) \notag \\
		=:    &\,I_1 + I_2 + I_3 + I_4. \label{eq:start of error estimate}
	\end{align}
	From interpolation error estimates and \pref{prop:consistency error} it follows that
	\begin{equation} \label{eq:error 1}
		|I_1| + |I_2| \le \frac\alpha8 \|v_h - u_h\|_{H^1(\Omega_h)}^2 + Ch^2.
	\end{equation}
	
	For $I_3$, we observe that
	\begin{align*}
		I_3 &= b_h(v_h - \tilde u + \tilde u - u_h, p_h - q_h + q_h - \tilde p) + d_h(p_h, p_h - q_h) - d_h(q_h, p_h - q_h) \\
		      &= b_h(v_h - \tilde u, q_h - \tilde p) + b_h(v_h - \tilde u, p_h - q_h) + b_h(\tilde u - u_h, q_h - \tilde p) \\
		      &\qquad + b_h(\tilde u - u_h, p_h - q_h) + d_h(p_h, p_h - q_h) \\
		      &\qquad - d_h(q_h, p_h - q_h) =: I_{31} + I_{32} + I_{33} + I_{34} + I_{35}.
	\end{align*}
	The first three terms are estimated as
	\begin{equation*}
		|I_{31}| \le Ch^2, \quad |I_{32}| \le Ch\|p_h - q_h\|_{L^2(\Omega_h)}, \quad |I_{33}| \le Ch^2 + \frac\alpha8 \|v_h - u_h\|_{H^1(\Omega_h)}^2,
	\end{equation*}
	whereas we know that $|I_{34}| \le Ch\|p_h - q_h\|_{L^2(\Omega_h)}$ by \pref{prop:consistency error}.
	$I_{35}$ is bounded, thanks to H\"older's inequality, by
	\begin{equation*}
		|I_{35}| \le \eta h^2 \|\nabla q_h\|_{L^2(\Omega_h)}^2 + \frac1{4} d_h(p_h - q_h, p_h - q_h) \le Ch^2 + \frac1{4} d_h(p_h - q_h, p_h - q_h).
	\end{equation*}
	
	To obtain further estimates of $I_{32}$ and $I_{34}$, we observe from \eref{eq:uniform inf-sup condition} that
	\begin{align*}
		     &C\|p_h - q_h\|_{L^2(\Omega_h)} \le \sup_{v_h\in\mathring V_h} \frac{b_h(v_h, p_h - q_h)}{\|v_h\|_{H^1(\Omega_h)}} + C\eta h \|\nabla(p_h - q_h)\|_{L^2(\Omega_h)} \\
		\le & \sup_{v_h\in\mathring V_h} \frac{b_h(v_h, p_h - \tilde p)}{\|v_h\|_{H^1(\Omega_h)}} + \sup_{v_h\in\mathring V_h} \frac{b_h(v_h, \tilde p - q_h)}{\|v_h\|_{H^1(\Omega_h)}} + C\eta h \|\nabla(p_h - q_h)\|_{L^2(\Omega_h)}.
	\end{align*}
	Here, the second term on the right-hand side is bounded by $Ch \|p\|_{H^1(\Omega)} = Ch$.
	We claim that the first term is bounded by $Ch + C\|v_h - u_h\|_{H^1(\Omega_h)}$.
	In fact, it follows from \eref{eq:consistency error}$_1$, in which $v_h$ is restricted to $\mathring V_h$ (hence $v_h = 0$ on $\Gamma_h$ so that $c_h(v_h\cdot n_h, \cdot) = 0$), that
	\begin{equation*}
		\sup_{v_h\in\mathring V_h} \frac{|a_h(\tilde u - u_h, v_h) + b_h(v_h, \tilde p - p_h)|}{\|v_h\|_{H^1(\Omega_h)}} \le Ch.
	\end{equation*}
	This combined with $\sup\limits_{v_h\in V_h} \frac{|a_h(\tilde u - u_h, v_h)|}{\|v_h\|_{H^1(\Omega_h)}} \le C\|\tilde u - u_h\|_{H^1(\Omega_h)} \le Ch + C\|v_h - u_h\|_{H^1(\Omega_h)}$ proves the claim.
	Consequently, we have
	\begin{equation}
		C\|p_h - q_h\|_{L^2(\Omega_h)} \le Ch + C\|v_h - u_h\|_{H^1(\Omega_h)} + C\eta h \|\nabla(p_h - q_h)\|_{L^2(\Omega_h)}. \label{eq:qh - ph}
	\end{equation}
	Collecting the above estimates for $I_{31}, \dots, I_{35}$ and noting that $\eta h^2 \|\nabla(p_h - q_h)\|_{L^2(\Omega_h)}^2 = d_h(p_h - q_h, p_h - q_h)$, we deduce
	\begin{equation} \label{eq:error 2}
		|I_3| \le Ch^2 + \frac\alpha4 \|v_h - u_h\|_{H^1(\Omega_h)}^2 + \frac12 d_h(p_h - q_h, p_h - q_h).
	\end{equation}
	
	In the same way as we computed $I_3$, one has
	\begin{align*}
		I_4 &= c_h((v_h - \tilde u)\cdot n_h, \mu_h - \tilde\lambda) + c_h((v_h - \tilde u)\cdot n_h, \lambda_h - \mu_h) \\
		      &\hspace{4.15cm} + c_h((\tilde u - u_h)\cdot n_h, \mu_h - \tilde\lambda) \\
		      &\quad +c_h((\tilde u - u_h)\cdot n_h, \lambda_h - \mu_h) + \epsilon c_h(\lambda_h, \lambda_h - \mu_h) \\
		      &\quad -\epsilon c_h(\mu_h, \lambda_h - \mu_h) =: I_{41} + I_{42} + I_{43} + I_{44} + I_{45}.
	\end{align*}
	By interpolation error estimates on $\Gamma_h$, we have
	\begin{align*}
		|I_{41}| &\le Ch^2, \quad |I_{42}| \le Ch^{3/2}\|\lambda_h - \mu_h\|_{L^2(\Gamma_h)}, \quad |I_{43}| \le Ch^{1/2}\|\tilde u - u_h\|_{H^1(\Omega_h)}, \\
		|I_{45}| &\le C\epsilon \|\lambda_h - \mu_h\|_{L^2(\Gamma_h)},
	\end{align*}
	whereas it follows from \eref{eq:consistency error}$_3$ that
	\begin{equation*}
		|I_{44}| \le Ch\|\lambda_h - \mu_h\|_{L^2(\Gamma_h)}.
	\end{equation*}
	Noting that $\|\lambda_h - \mu_h\|_{L^2(\Gamma_h)}^2 = c_h(\lambda_h - \mu_h, \lambda_h - \mu_h)$ and using Young's inequality, we arrive at
	\begin{align}
		|I_4| &\le Ch^2 + Ch^3/\epsilon + Ch + C\epsilon + Ch^2/\epsilon \notag \\
		        &\hspace{2.9cm} + \frac\alpha8 \|v_h - u_h\|_{H^1(\Omega_h)}^2 + \frac\epsilon2 c_h(\lambda_h - \mu_h, \lambda_h - \mu_h) \notag \\
			&\le C(h + \epsilon + h^2/\epsilon) + \frac\alpha8 \|v_h - u_h\|_{H^1(\Omega_h)}^2 + \frac\epsilon2 c_h(\lambda_h - \mu_h, \lambda_h - \mu_h). \label{eq:error 3}
	\end{align}
	
	Combining \eref{eq:error 1}, \eref{eq:error 2}, and \eref{eq:error 3} with \eref{eq:start of error estimate}, we conclude the desired estimate for $\|v_h - u_h\|_{H^1(\Omega_h)}$.
	The result for $\|q_h - p_h\|_{L^2(\Omega_h)}$ follows from \eref{eq:qh - ph}.
	This completes the proof of \tref{thm:error estimate}.
\end{proof}

\begin{rem}
	According to the theorem, the best rate of convergence is $O(h^{1/2})$ obtained by choosing $\epsilon = O(h)$, which is not optimal.
	Let us highlight the reasons for this sub-optimality.
	First, as far as the variational principle is concerned, the most suitable regularity to work with for $\lambda_h$ would be $H^{-1/2}(\Gamma_h)$, instead of $L^2(\Gamma_h)$ as presented above.
	However, it is not possible to extract this regularity from $I_{44}$ above (more precisely, $I_7$ in the proof of \pref{prop:consistency error}) because $n_h\notin H^{1/2}(\Gamma_h)$.
	Second, it is not trivial whether the following inf-sup condition would hold:
	\begin{equation} \label{eq:inf-sup condition for lambdah}
		C \|\mu_h\|_{H^{-1/2}(\Gamma_h)} \le \sup_{v_h\in V_h} \frac{c_h(v_h\cdot n_h, \mu_h)}{\|v_h\|_{H^1(\Omega_h)}} \qquad \forall \mu_h\in \Lambda_h.
	\end{equation}
	In the case $\Omega_h = \Omega$, this condition is valid for a suitable choice of $\Lambda_h$.
	\c{C}a\u{g}lar and Liakos \cite{Cag04,CaLi09} took advantage of this fact to derive the optimal rate of convergence $O(h + \epsilon)$.
\end{rem}

\section{Penalty FE scheme with reduced-order numerical integration} \label{sec:reduced-order integration}
In this section, we investigate problem \eref{eq:discrete problem with 2 variables} in which $c_h$ is replaced with its reduced-order numerical integration $c_h^1$ defined via the midpoint (barycenter) formula as follows:
\begin{equation*}
	c_h^1(\lambda, \mu) := \sum_{S\in\mathcal S_h} |S| \lambda(m_S) \mu(m_S), \quad \lambda,\mu\in C(\Gamma_h),
\end{equation*}
where $|S|$ denotes the area of $S$ and $m_S$ is the midpoint of $S$ when $N=2$ (resp., the barycenter of $S$ when $N=3$).
Because we exploit pointwise evaluation of functions, we assume higher regularity of the exact solutions as follows:
\begin{equation*}
	u\in W^{2,\infty}(\Omega)^N, \quad p\in W^{1,\infty}(\Omega), \quad \lambda\in W^{1,\infty}(\Gamma),
\end{equation*}
which implies $f\in L^\infty(\Omega)^N,\, g\in W^{2,\infty}(\Gamma),\, \tau\in W^{1,\infty}(\Gamma)^N$.

Then the problem we propose reads: find $(u_h, p_h)\in V_h\times Q_h$ such that
\begin{equation} \label{eq:discrete problem with reduced-order quadrature}
	\left\{
	\begin{aligned}
		a_h(u_h, v_h) + b_h(v_h, p_h) + \frac1\epsilon c_h^1(u_h\cdot n_h - \tilde g, v_h\cdot n_h) &= (\tilde f, v_h)_{\Omega_h} + (\tilde\tau, v_h)_{\Gamma_h} \\
		b_h(u_h, q_h) &= d_h(p_h, q_h)
	\end{aligned}
	\right.
\end{equation}
for all $(v_h, q_h)\in V_h\times Q_h$.
The well-posedness of this problem is obtained by the same manner as in \pref{prop:discrete well-posedness}.
We also find that its solution satisfies the following three-variable formulation as we derived \eref{eq:discrete problem with 3 variables}:
\begin{equation*}
	\left\{
	\begin{aligned}
		a_h(u_h, v_h) + b_h(v_h, p_h) +c_h^1(v_h\cdot n_h, \lambda_h) &= (\tilde f, v_h)_{\Omega_h} + (\tilde\tau, v_h)_{\Gamma_h} & \forall v_h\in V_h, \\
		b_h(u_h, q_h) &= d_h(p_h, q_h) & \forall q_h\in Q_h, \\
		c_h^1(u_h\cdot n_h - \tilde g, \mu) &= \epsilon c_h^1(\lambda_h, \mu) &\hspace{-3mm} \forall \mu\in C(\Gamma_h),
	\end{aligned}
	\right.
\end{equation*}
where $\lambda_h$ is defined only on $\{m_S \,:\, S\in\mathcal S_h\}$ by $\lambda_h(m_S) = \frac1\epsilon (u_h\cdot n_h - \tilde g)|_{m_S}$.
Likewise, the error analysis is mostly parallel to the arguments in \tref{thm:error estimate}.
Thereby, in the sequel we only focus on what will change due to the replacement of $c_h$ by $c_h^1$.
In doing so, first we observe that:

\begin{lem}
	Let $v_h\in V_h$ and $\tilde\lambda\in W^{1,\infty}(\tilde\Omega)$. Then
	\begin{equation*}
		|c_h(v_h\cdot n_h, \tilde\lambda) - c_h^1(v_h\cdot n_h, \tilde\lambda)| \le h\|v_h\|_{L^1(\Gamma_h)}\|\tilde\lambda\|_{W^{1,\infty}(\tilde\Omega)}.
	\end{equation*}
\end{lem}
\begin{proof}
	Since the midpoint (barycenter) formula is exact for affine functions, one obtains
	\begin{equation*}
		c_h(v_h\cdot n_h, \tilde\lambda) - c_h^1(v_h\cdot n_h, \tilde\lambda) = \sum_{S\in\mathcal S_h} \int_S v_h\cdot n_h(\tilde\lambda - \tilde\lambda(m_S))\,d\gamma_h.
	\end{equation*}
	This combined with $\|\tilde\lambda - \tilde\lambda(m_S)\|_{L^\infty(S)} \le \|\tilde\lambda\|_{W^{1,\infty}(\tilde\Omega)}h$
	(note that $\mathrm{diam}\,S \le h$) concludes the desired result.
\end{proof}
Combining this lemma with the estimates of $I_5$ in the proof of \pref{prop:consistency error}, we see that \eref{eq:consistency error}$_1$ remains the same even if we replace $c_h$ by $c_h^1$.

Next we consider the analysis of \eref{eq:consistency error}$_3$, namely, the estimates for $I_7$--$I_{10}$ in the proof of \pref{prop:consistency error}.
To this end we introduce a semi-norm in $\Lambda_h$ by 
\begin{equation*}
	|\mu_h|_{\Lambda_h} := c_h^1(\mu_h, \mu_h)^{1/2}.
\end{equation*}
By \lref{lem:n and nh} and Cauchy-Schwarz inequality, $I_7$ is bounded by $Ch^2 |\mu_h|_{\Lambda_h}$ if $N=2$ and by $Ch |\mu_h|_{\Lambda_h}$ if $N=3$.
By the regularity assumption $\tilde u, \tilde g\in W^{1,\infty}(\tilde\Omega)$ and by \pref{prop:estimate of Phi}, we have $|I_8|+|I_9| \le Ch^2 |\mu_h|_{\Lambda_h}$.
We notice that $I_{10}$ in the present situation is zero. Therefore, instead of \eref{eq:consistency error}$_3$ we obtain
\begin{equation} \label{eq:consistency error for ch1}
	|c_h^1((\tilde u - u_h)\cdot n_h, \mu_h) + \epsilon c_h^1(\lambda_h, \mu_h)| \le Ch^j |\mu_h|_{\Lambda_h} \qquad \forall \mu_h\in\Lambda_h,
\end{equation}
where $j=2$ if $N=2$ and $j=1$ if $N=3$.

Finally, we consider the estimates of $I_4$ in the proof of \tref{thm:error estimate}.
This time we may choose $\mu_h = I_h\tilde\lambda$ as an interpolation of $\tilde\lambda$.
In view of the regularity assumption $\tilde u\in W^{2,\infty}(\tilde\Omega)$ and $\tilde\lambda\in W^{1,\infty}(\tilde\Omega)$ and by virtue of \lref{lem:n and nh}, we have
\begin{align*}
	|I_{41}| &\le Ch^3, \quad |I_{42}| \le Ch^2|\lambda_h - \mu_h|_{\Lambda_h}, \quad |I_{43}| \le Ch\|\tilde u - u_h\|_{H^1(\Omega_h)}, \\
	|I_{45}| &\le C\epsilon |\lambda_h - \mu_h|_{\Lambda_h}.
\end{align*}
By \eref{eq:consistency error for ch1}, $I_{44} \le Ch^j |\lambda_h - \mu_h|_{\Lambda_h}$.
Consequently, instead of \eref{eq:error 3} we obtain
\begin{equation*}
	|I_4| \le C(h^2 + \epsilon + h^{2j}/\epsilon) + \frac\alpha8 \|v_h - u_h\|_{H^1(\Omega_h)}^2 + \frac\epsilon2 c_h^1(\lambda_h - \mu_h, \lambda_h - \mu_h).
\end{equation*}
From these observations, we arrive at the following result.

\begin{thm}
	In addition to the hypotheses of \tref{thm:error estimate} we assume that the solution $(u, p)$ of \eref{eq:weak form with constraint} possesses the $W^{2,\infty}(\Omega)^N\times W^{1,\infty}(\Omega)$-regularity.
	Let $(u_h, p_h)$ be the solution of \eref{eq:discrete problem with reduced-order quadrature}.
	Then there exists a constant $C=C(N,\Omega,\nu,u,p)$, independent of $h$ and $\epsilon$, such that
	\begin{equation} \label{eq:error bound for reduced scheme}
		\textstyle
		\|\tilde u - u_h\|_{H^1(\Omega_h)} + \|(\tilde p+k_h) - p_h\|_{L^2(\Omega_h)} \le C(h + \sqrt\epsilon + \frac{h^j}{\sqrt\epsilon}),
	\end{equation}
	where $j=2$ if $N=2$ and $j=1$ if $N=3$.
\end{thm}
\begin{rem}
	According to the theorem, choosing $\epsilon = O(h^2)$ gives us the optimal rate of convergence $O(h)$ when $N=2$.
	When $N=3$, at least we see that introduction of reduced-order numerical integration does not deteriorate the rate of convergence.
	Our numerical example given in the next section shows that it does improve the accuracy for $N=3$ as well.
\end{rem}

\section{Numerical examples} \label{sec:numerical examples}
In the sequel, we refer to the schemes \eref{eq:discrete problem with 2 variables} and \eref{eq:discrete problem with reduced-order quadrature}, i.e.\ without and with reduced-order numerical integration, as ``non-reduced'' and ``reduced'', respectively.
\subsection{Two-dimensional test}
In this example, all the computations are done with the use of \verb#FreeFem++# \cite{freefem} choosing the P1/P1 element, i.e.\ $l=1$, together with $\eta=0.01$.
Let $\Omega$ be the unit disk, namely, $\Omega = \{ (x,y)\in\mathbb R^2 \,:\, x^2+y^2<1 \}$.
We consider the slip boundary value problem \eref{eq:Stokes slip BC} for $\nu = 1$ and for $f,g,\tau$ given by
\begin{align*}
	f &= \begin{pmatrix} -y(x^2+y^2)+16y \\ x(x^2+y^2) \end{pmatrix}, \quad g=0, \\
	\tau &= \begin{pmatrix} 1-x^2 & -xy \\ -xy & 1-y^2 \end{pmatrix} \begin{pmatrix} -12xy & 2(x^2-y^2) \\ 2(x^2-y^2) & -4xy \end{pmatrix} \begin{pmatrix}x\\y\end{pmatrix},
\end{align*}
in which case we have the analytical solution $u=( -y(x^2+y^2), x(x^2+y^2) )^T,\, p=8xy$.
We also introduce the solution of the no-slip boundary value problem denoted by $(u^{\textrm{no-slip}}, p^{\textrm{no-slip}})$.
Namely, it is determined according to the same $f$ as above and to the boundary condition $u^{\textrm{no-slip}} = 0$ on $\Gamma$.
\fref{fig:velocity profile of slip and no-slip} shows the velocity profiles of the two solutions; one notices the clear difference in their circulating directions and in the maximum modulus of velocity.

\begin{figure}[htbp]
	\centering
	\includegraphics[width=7.7cm]{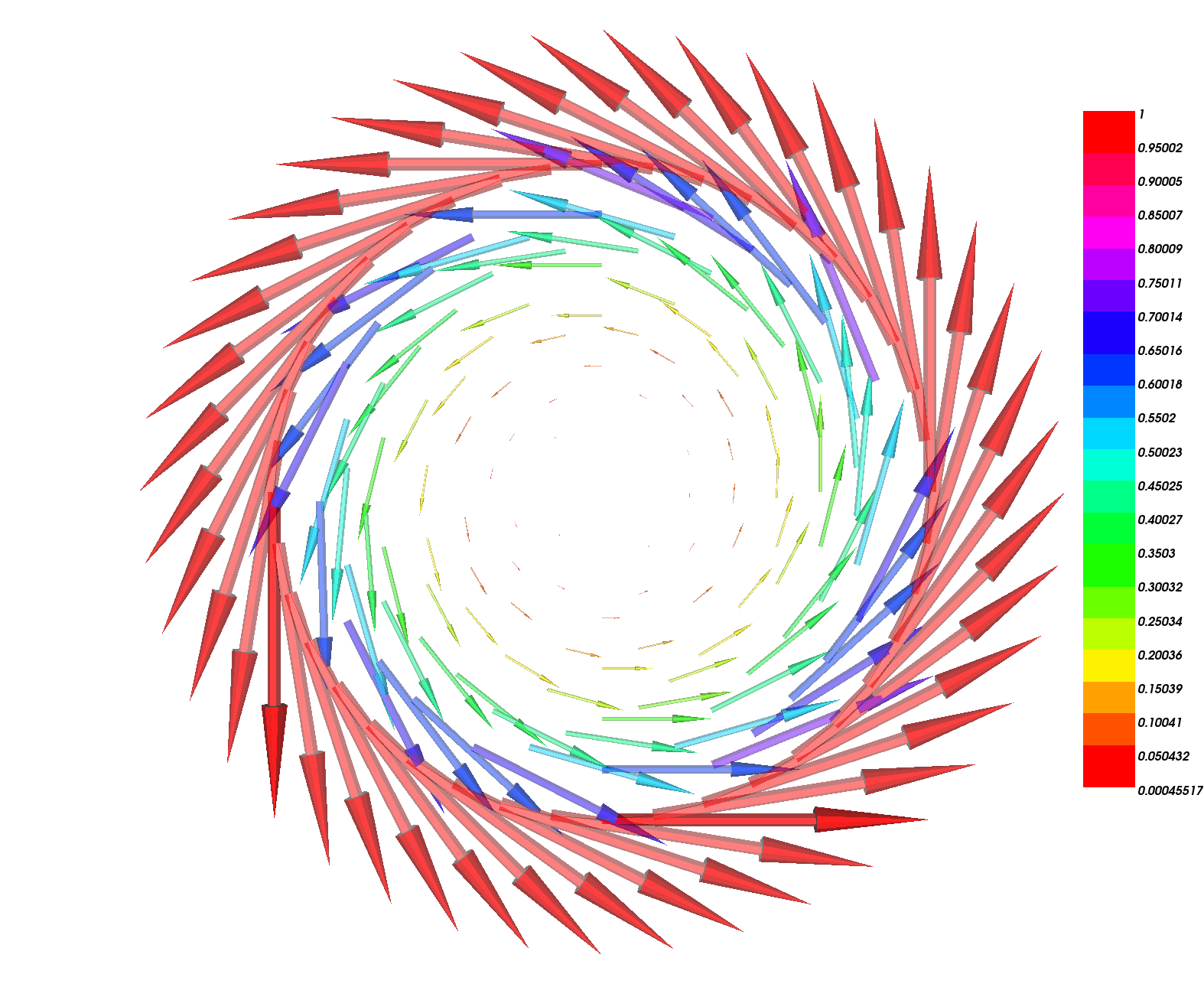}
	\includegraphics[width=7.7cm]{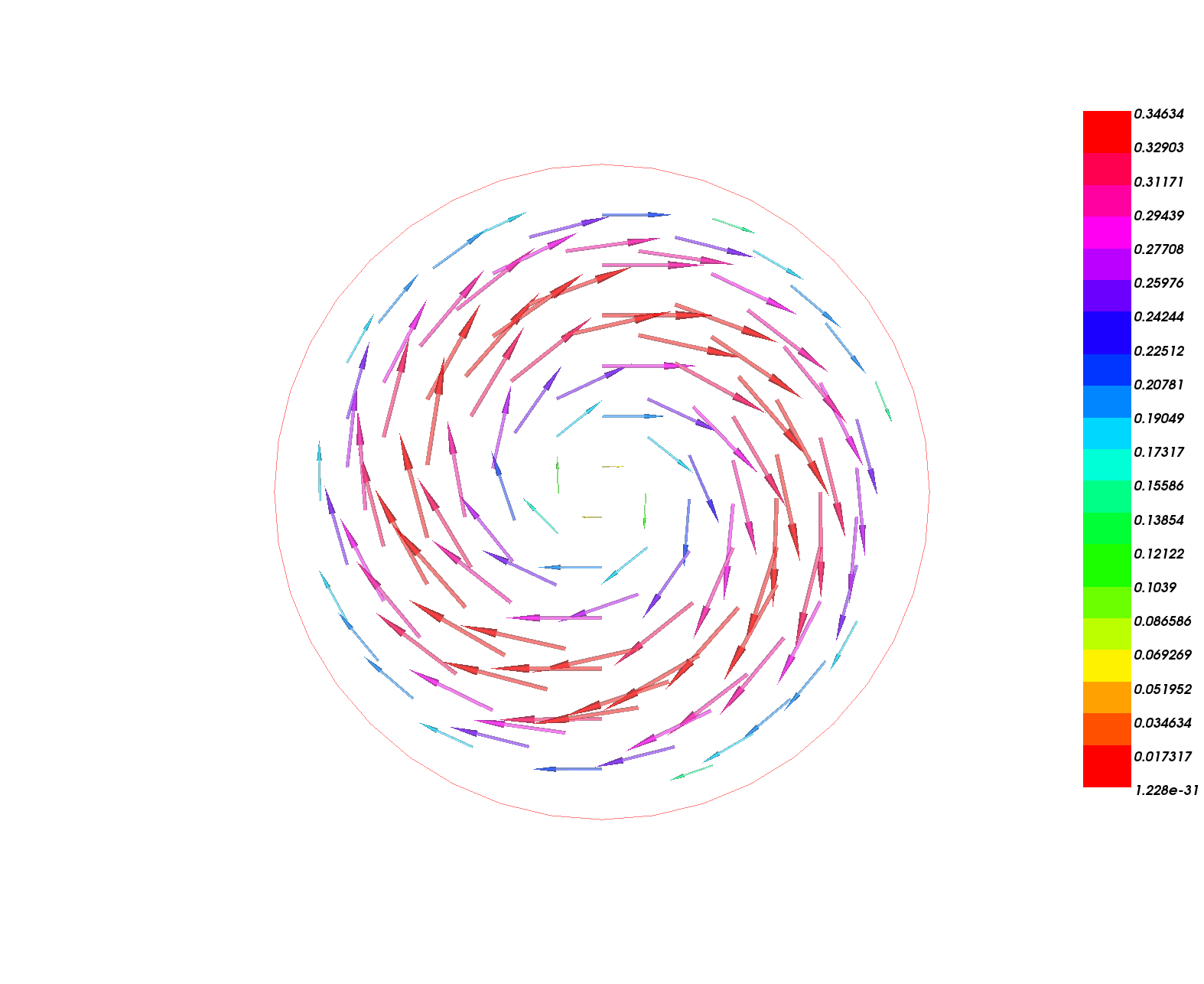}
	\caption{Velocity profiles of $u$ (left) and $u^{\textrm{no-slip}}$ (right).}
	\label{fig:velocity profile of slip and no-slip}
\end{figure}
\begin{figure}[htbp]
	\centering
	\includegraphics[width=7.7cm]{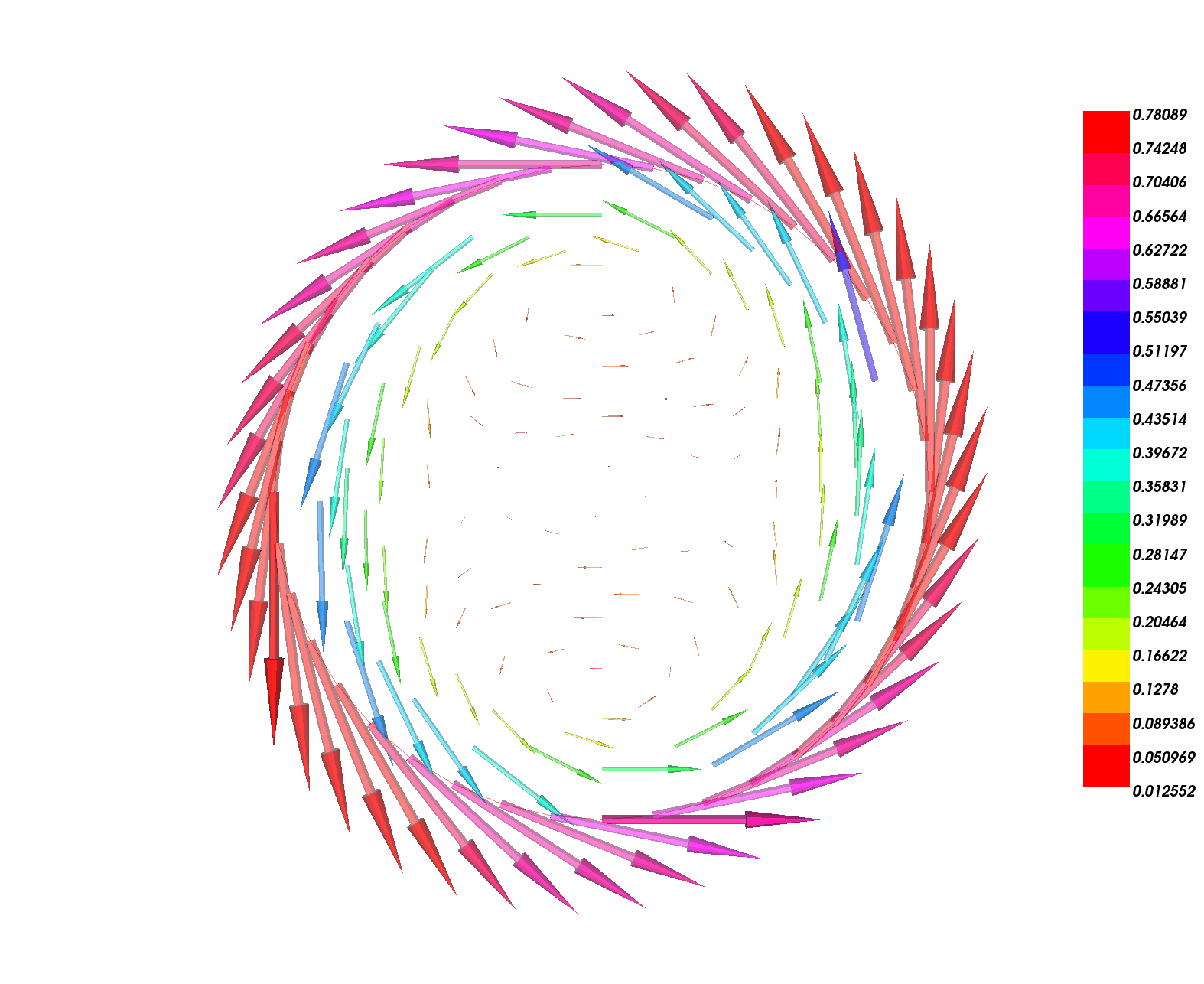}
	\includegraphics[width=7.7cm]{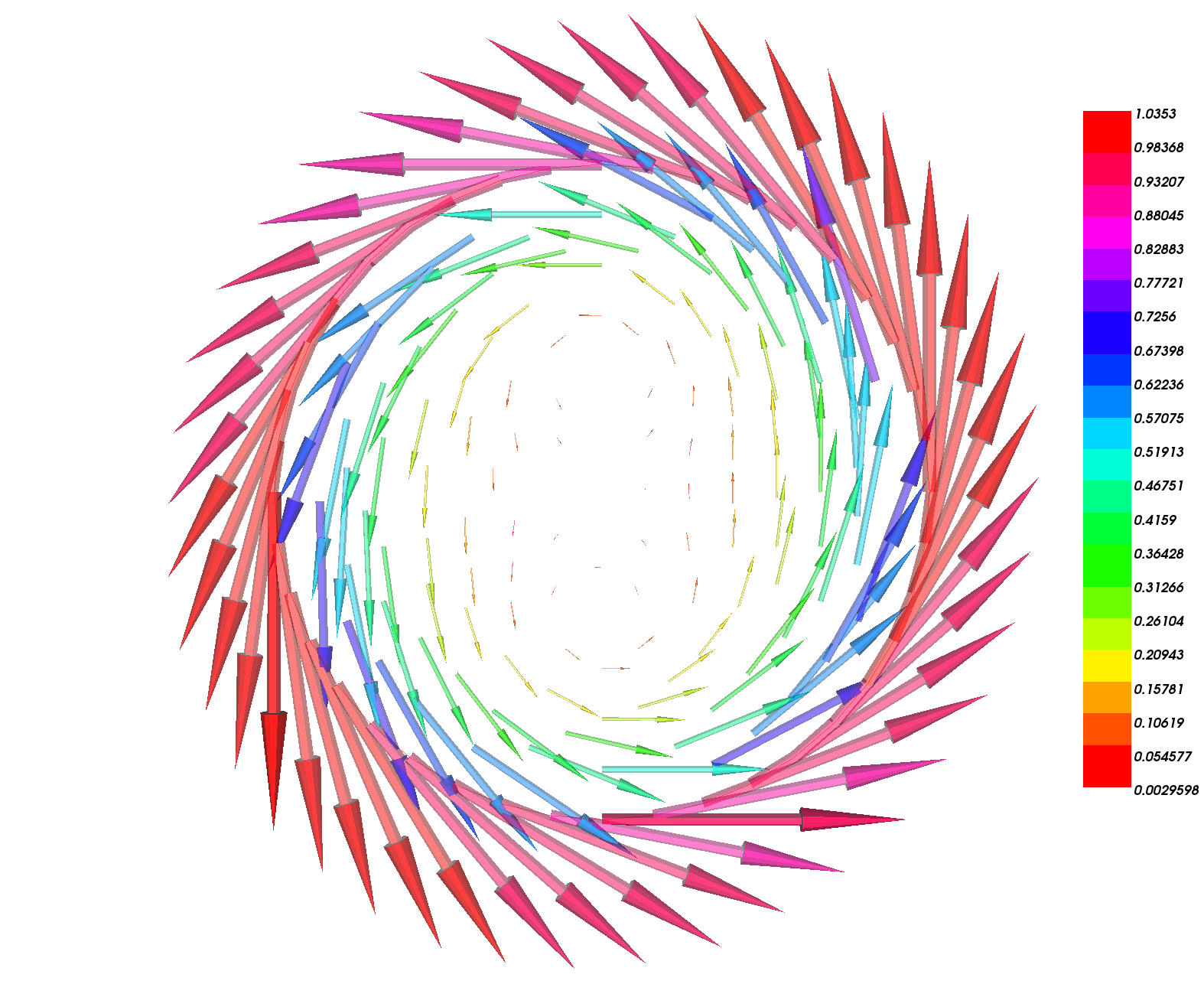}
	\includegraphics[width=7.7cm]{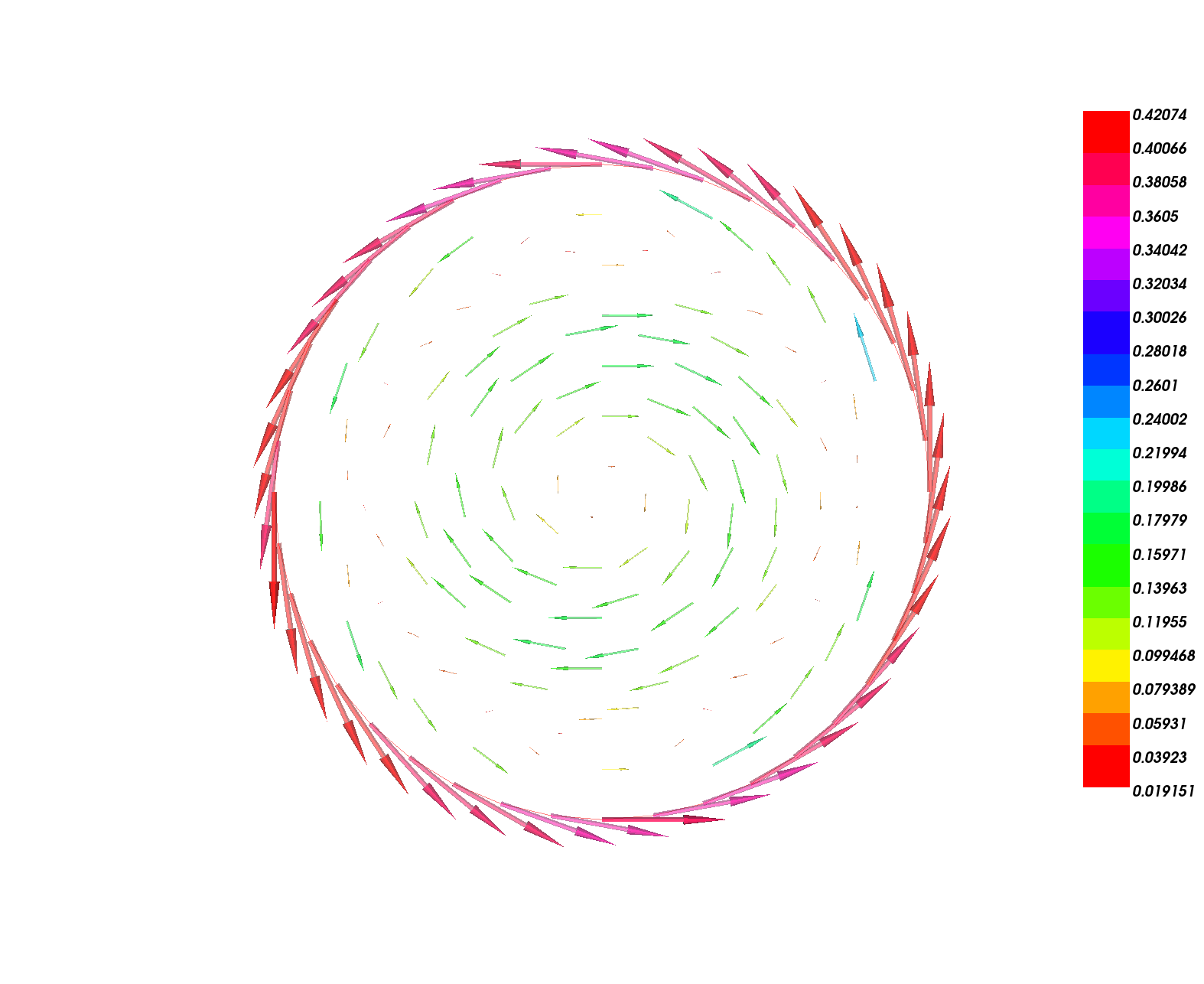}
	\includegraphics[width=7.7cm]{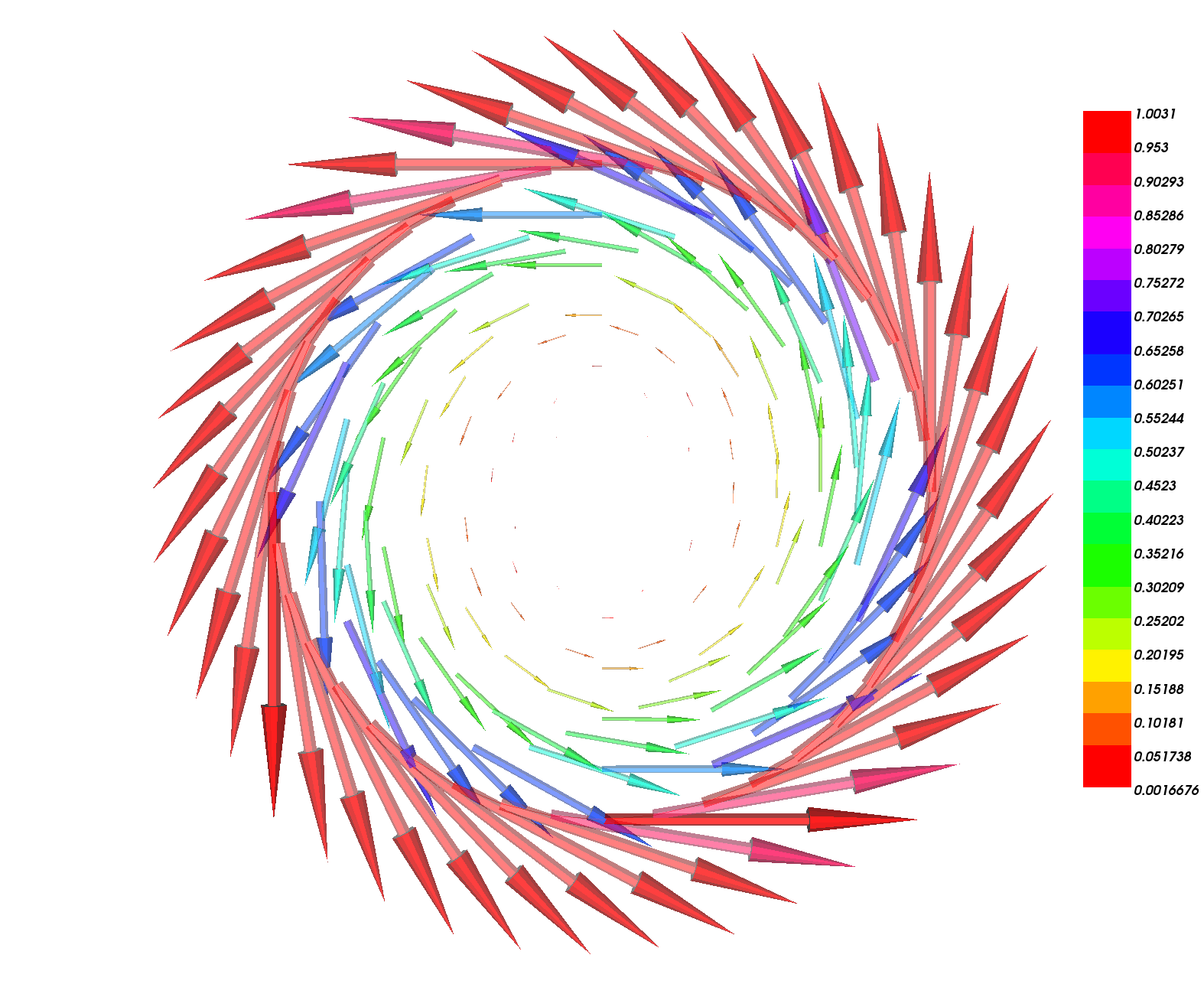}
	\includegraphics[width=7.7cm]{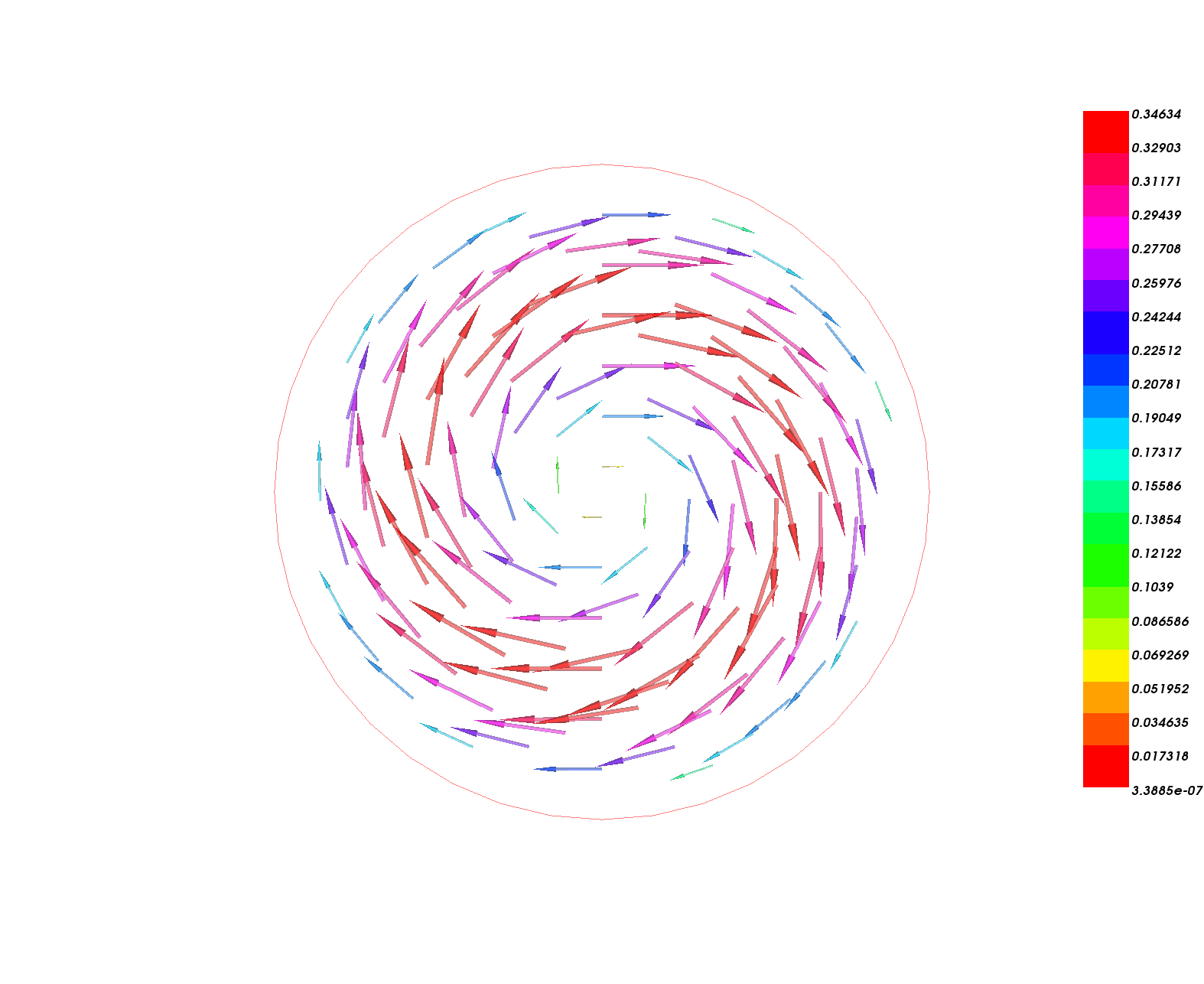}
	\includegraphics[width=7.7cm]{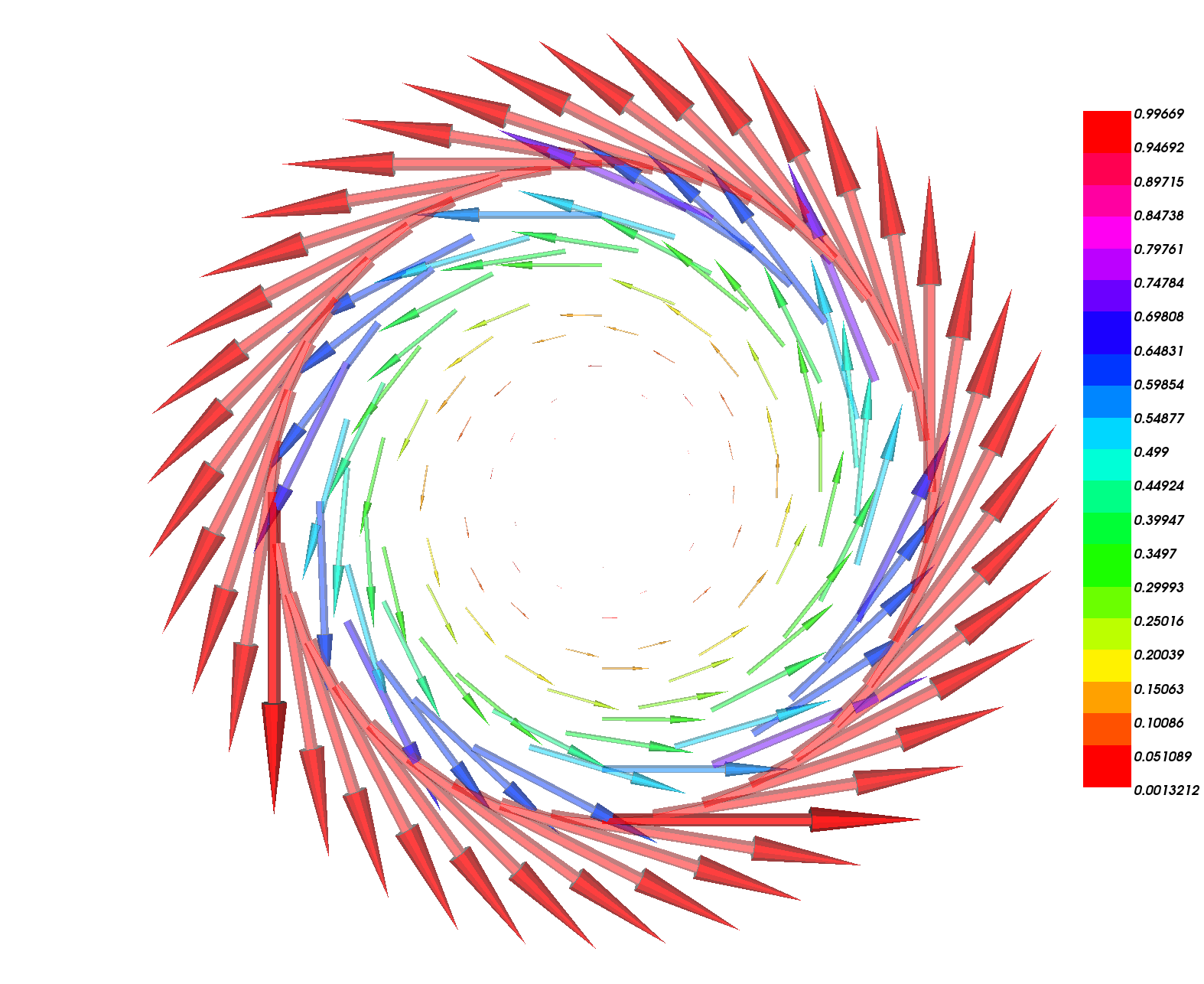}
	\caption{Velocity profiles of numerical solutions computed with non-reduced (left column) and reduced (right column) schemes.
	For each row (top to bottom), $\epsilon$ is chosen as $0.1h$, $0.1h^2$, $10^{-8}$, where $h\approx0.241$.}
	\label{fig:velocity profile of numerical solutions}
\end{figure}
\begin{table}[htbp]
	\centering
	\caption{Convergence behavior of velocity in the $H^1(\Omega_h)^2$-norm in the 2D test (top: $\epsilon = 0.1h$, bottom: $\epsilon=0.1h^2$). DOF means the number of degrees of freedom.}
	\label{tab:convergence behavior in 2D}
	\begin{tabular}{crcccccc}
		$h$ & DOF & $\|u - u_h^{\textrm{NR}}\|$ & Rate & $\|u - u_h^{\textrm{R}}\|$ & Rate & $\|u - u_h^{\textrm{Dir}}\|$ & Rate \\
		\toprule
		0.316 &       333 & 1.043 & ---    & 0.575 & ---    & 0.464 & ---     \\
		0.165 &     1182 & 0.567 & 0.94 & 0.296 & 1.09 & 0.231 & 1.07 \\
		0.078 &     4488 & 0.310 & 0.81 & 0.145 & 0.94 & 0.114 & 0.94 \\
		0.045 &   17391 & 0.146 & 1.37 & 0.077 & 1.30 & 0.057 & 1.28 \\
		0.023 &   69270 & 0.074 & 1.02 & 0.039 & 1.02 & 0.028 & 1.02 \\
		0.012 & 274956 & 0.036 & 1.16 & 0.020 & 1.05 & 0.014 & 1.10
	\end{tabular}
	
	\vspace{2mm}
	\begin{tabular}{crcccccc}
		$h$ & DOF & $\|u - u_h^{\textrm{NR}}\|$ & Rate & $\|u - u_h^{\textrm{R}}\|$ & Rate & $\|u - u_h^{\textrm{Dir}}\|$ & Rate \\
		\toprule
		0.316 &       333 & 1.683 & ---        & 0.479 & ---    & 0.464 & ---     \\
		0.165 &     1182 & 1.559 & 0.12     & 0.232 & 1.12 & 0.231 & 1.07 \\
		0.078 &     4488 & 1.618 & ($<0$) & 0.114 & 0.95 & 0.114 & 0.94 \\
		0.045 &   17391 & 1.412 & 0.25     & 0.057 & 1.28 & 0.057 & 1.28 \\
		0.023 &   69270 & 1.388 & 0.03     & 0.028 & 1.02 & 0.028 & 1.02 \\
		0.012 & 274956 & 1.293 & 0.11     & 0.014 & 1.10 & 0.014 & 1.10
	\end{tabular}
\end{table}
\begin{figure}[htbp]
	\centering
	\includegraphics[width=5cm]{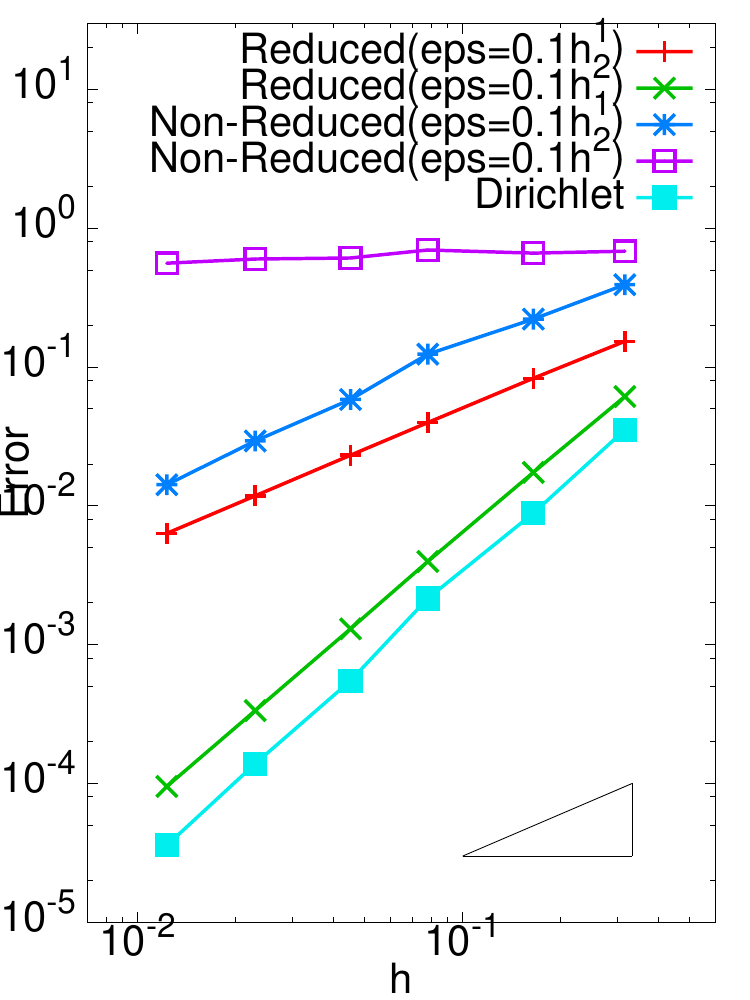}
	\includegraphics[width=5cm]{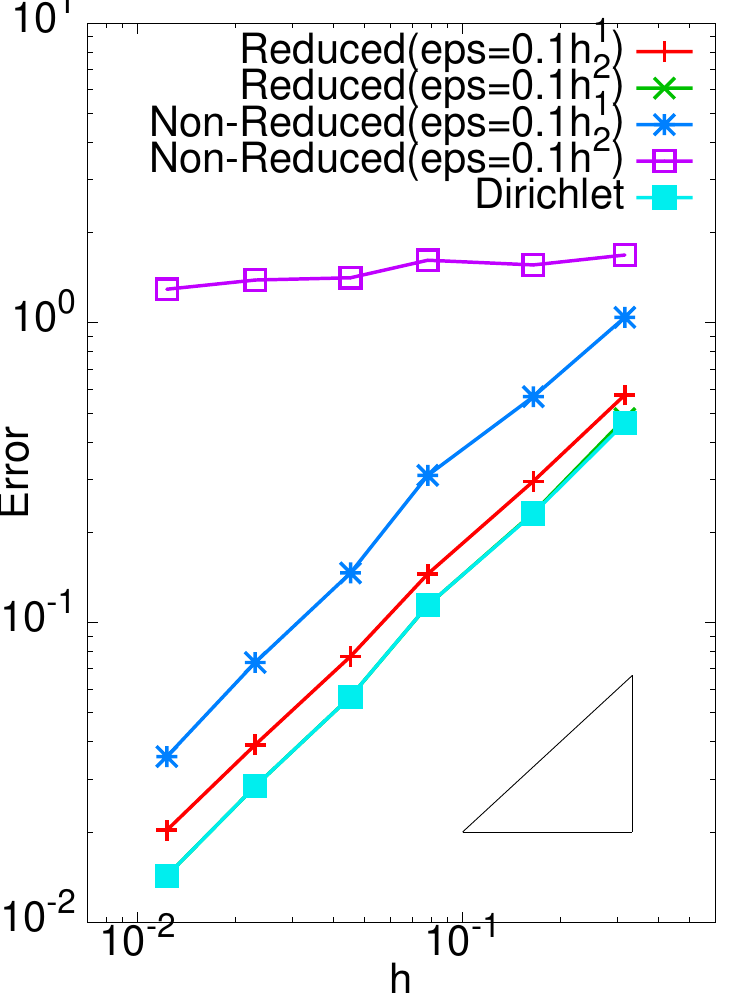}
	\includegraphics[width=5cm]{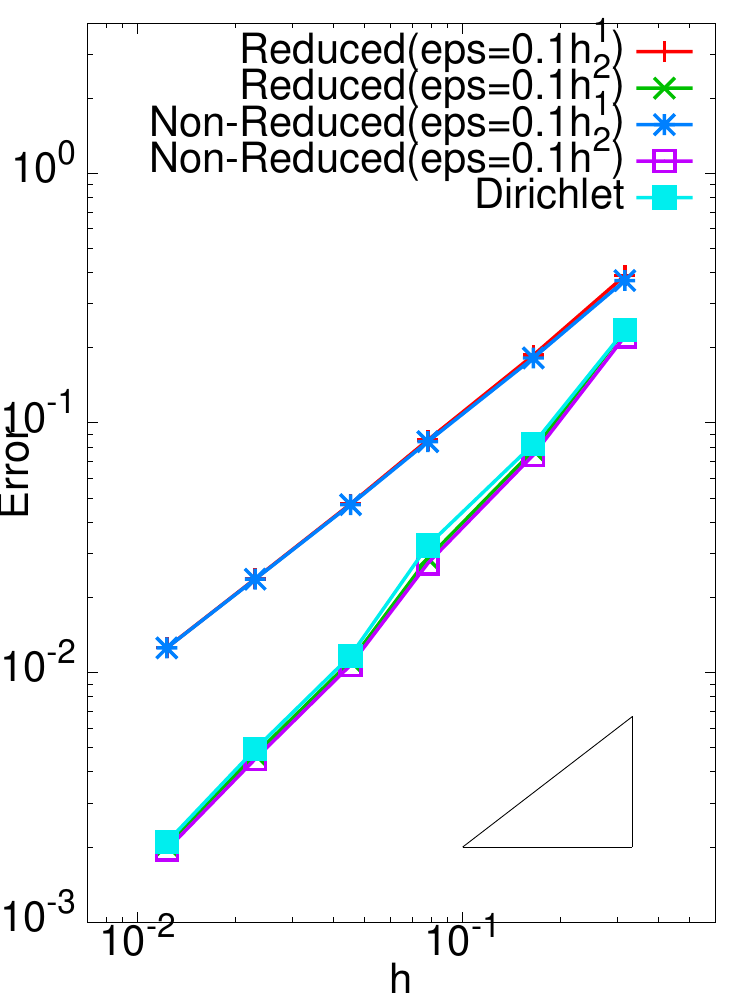}
	\caption{Convergence behavior of $\|u - u_h\|_{L^2(\Omega_h)}$ (left), $\|u - u_h\|_{H^1(\Omega_h)}$ (middle), $\|p - p_h\|_{L^2(\Omega_h)}$ (right) in the 2D test.
	The triangles indicate the slope $O(h)$.}
	\label{fig:convergence behavior in 2D}
\end{figure}

On a mesh with $h\approx0.241$, we computed numerical solutions of the slip boundary value problem using the non-reduced/reduced schemes for three choices of the penalty parameter $\epsilon$: $O(h)$, $O(h^2)$, and very small.
The results are shown in \fref{fig:velocity profile of numerical solutions}.
We find that the reduced scheme gives more robust and accurate approximate solutions.
In fact, in case $\epsilon = O(h)$, there are two spurious circulations inside $\Omega$ for the non-reduced scheme; we remark that refining a mesh and keeping $\epsilon=O(h)$ did not suppress them.
For smaller $\epsilon$ the non-reduced scheme fails to capture the correct solution and seems to approach the no-slip boundary value problem.
This behavior is somehow expected because letting $\epsilon\to0$ in the penalty term of \eref{eq:discrete problem with 2 variables} implies (at least formally) the constraint $u_h\cdot n_h = 0$ on $\Gamma_h$, which undesirably collapses into $u_h=0$ on $\Gamma_h$ as observed in \sref{sec:intro}.
However, the reduced scheme produces solutions which capture the slip boundary condition correctly for all $\epsilon>0$ sufficiently small.
Therefore, it is expected that the error bound \eref{eq:error bound for reduced scheme} could be improved in such a way that the reciprocal of $\epsilon$ would not appear (we conjecture that some inf-sup condition like \eref{eq:inf-sup condition for lambdah} would be valid).

Next we study the convergence property of the non-reduced and reduced schemes, whose solutions are denoted by $(u_h^{\textrm{NR}}, p_h^{\textrm{NR}})$ and $(u_h^{\textrm{R}}, p_h^{\textrm{R}})$, respectively.
We compare the convergence behavior of them with that of the numerical solutions computed with the Dirichlet boundary condition, which are denoted by $(u_h^{\textrm{Dir}}, p_h^{\textrm{Dir}})$ (i.e., the boundary condition $u_h^{\textrm{Dir}} = u$ on $\Gamma_h$ is imposed).
The linear solver is chosen as UMFPACK, and $u - u_h$ and $p - p_h$ are interpolated into the quadratic finite element space to compute errors in the associated norms.
For convenience, we also report errors in the $L^2(\Omega_h)^2$-norm of velocity, and we remark that $\|u\|_{L^2(\Omega)}\approx0.886$, $\|u\|_{H^1(\Omega)}\approx3.355$, $\|p\|_{L^2(\Omega)}\approx2.894$.
The results are presented in \tabref{tab:convergence behavior in 2D} and \fref{fig:convergence behavior in 2D} for two choices of $\epsilon$.
We find that $\|u - u_h^{\textrm{R}}\|_{L^2(\Omega_h)} = O(h^2)$ for $\epsilon = O(h^2)$ and that $\|u - u_h^{\bullet}\|_{H^1(\Omega_h)} = O(h)$ for $\bullet=\textrm{NR, R}$ and for $\epsilon = O(h)$, which cannot be explained by the theory.
Nevertheless, the fact that the reduced scheme with $\epsilon = O(h^2)$ achieves the best accuracy is in accordance with the theoretical prediction.
We see from \fref{fig:convergence behavior in 2D} that its accuracy in the energy norm is almost the same as that of $(u_h^{\textrm{Dir}}, p_h^{\textrm{Dir}})$.

\subsection{Three-dimensional test}
In this example, we focus on the verification of convergence.
Let $\Omega$ be the unit sphere, i.e., $\Omega = \{ (x,y,z)\in\mathbb R^3 \,:\, x^2+y^2+z^2<1 \}$.
We consider \eref{eq:Stokes slip BC} for $\nu = 1$ and for $f,g,\tau$ such that the analytical solution is
\begin{equation*}
	u = \begin{pmatrix} 10x^2yz(y - z) \\ 10y^2zx(z - x) \\ 10z^2xy(x - y) \end{pmatrix}, \quad p = 10xyz(x+y+z).
\end{equation*}
We remark that $\|u\|_{L^2(\Omega)}\approx0.708$, $\|u\|_{H^1(\Omega)}\approx4.943$, $\|p\|_{L^2(\Omega)}\approx1.043$.
The computations are done with the use of \verb#FEniCS# \cite{fenics} (combined with \verb+Gmsh+ \cite{gmsh} for obtaining meshes) choosing  the P1/P1 element with $\eta=0.1$.
The linear solver is GMRES, preconditioned by incomplete LU factorization, with the restart number 200 and with the relative tolerance $10^{-8}$.
As in the previous example, $u - u_h$ and $p - p_h$ are interpolated into the quadratic finite element space to compute their norms.
The results are reported in \tabref{tab:convergence behavior in 3D} and \fref{fig:convergence behavior in 3D}.
Except case of the non-reduced scheme with $\epsilon = O(h^2)$, the errors seem to converge at the rate $O(h)$, which is better than the theoretically predicted one $O(h^{1/2})$.
Our opinion is that the interior errors would be dominant with the resolution of meshes considered here and that the suboptimal rate $O(h^{1/2})$ would be observed only for a very fine mesh.
However, from the results we infer that the reduced-order numerical integration is also effective for the case $N=3$, in which the choice of $\epsilon = O(h^2)$ may be recommended (although this was not justified by a rigorous proof).
With this choice, as in the two-dimensional case, the accuracy in the energy norm is comparable with that of $(u_h^{\textrm{Dir}}, p_h^{\textrm{Dir}})$.

\begin{table}[tbp]
	\centering
	\caption{Convergence behavior of velocity in the $H^1(\Omega_h)^3$-norm in the 3D test (top: $\epsilon = 0.1h$, bottom: $\epsilon=0.1h^2$).
	Itr means the number of iterations required for GMRES to converge.}
	\label{tab:convergence behavior in 3D}
	\begin{tabular}{ccccccccccc}
		$h$ & DOF & $\|u - u_h^{\textrm{NR}}\|$ & Rate & Itr & $\|u - u_h^{\textrm{R}}\|$ & Rate & Itr & $\|u - u_h^{\textrm{Dir}}\|$ & Rate & Itr \\
		\toprule
		0.240 & 1.11E+4 & 1.471 & ---     & 69     & 1.048 & ---    & 70     & 1.268 & ---    & 76     \\
		0.113 & 1.12E+5 & 0.656 & 1.08 & 238   & 0.638 & 1.06 & 279   & 0.574 & 1.06 & 349   \\
		0.075 & 3.42E+5 & 0.454 & 0.89 & 352   & 0.448 & 0.86 & 352   & 0.405 & 0.85 & 393   \\
		0.062 & 6.59E+5 & 0.372 & 1.08 & 469   & 0.367 & 1.07 & 473   & 0.326 & 1.16 & 751   \\
		0.052 & 1.17E+6 & 0.306 & 1.05 & 655   & 0.303 & 1.04 & 658   & 0.266 & 1.10 & 790   \\
		0.045 & 1.88E+6 & 0.263 & 1.04 & 901   & 0.261 & 1.03 & 899   & 0.227 & 1.09 & 1979 \\
		0.039 & 2.39E+6 & 0.238 & 0.87 & 1749 & 0.236 & 0.87 & 1638 & 0.205 & 0.88 & 5179 \\
	\end{tabular}
	
	\vspace{2mm}
	\begin{tabular}{ccccccccccc}
		$h$ & DOF & $\|u - u_h^{\textrm{NR}}\|$ & Rate & Itr & $\|u - u_h^{\textrm{R}}\|$ & Rate & Itr & $\|u - u_h^{\textrm{Dir}}\|$ & Rate & Itr \\
		\toprule
		0.240 & 1.11E+4 & 1.590 & ---     & 77     & 1.350 & ---    & 81     & 1.268 & ---     & 76    \\
		0.113 & 1.12E+5 & 0.877 & 0.79 & 270   & 0.579 & 1.13 & 304   & 0.574 & 1.06 & 349   \\
		0.075 & 3.42E+5 & 0.630 & 0.80 & 467   & 0.405 & 0.87 & 646   & 0.405 & 0.85 & 393   \\
		0.062 & 6.59E+5 & 0.575 & 0.49 & 742   & 0.327 & 1.15 & 782   & 0.326 & 1.16 & 751   \\
		0.052 & 1.17E+6 & 0.534 & 0.41 & 1111  & 0.271 & 1.02 & 1348   & 0.266 & 1.10 & 790   \\
		0.045 & 1.88E+6 & 0.493 & 0.54 & 1735 & 0.231 & 1.08 & 2175   & 0.227 & 1.09 & 1979 \\
		0.039 & 2.39E+6 & 0.477 & 0.29 & 2103 & 0.201 & 0.89 & 2600 & 0.205 & 0.88 & 5179 \\
	\end{tabular}
\end{table}
\begin{figure}[htbp]
	\centering
	\includegraphics[width=5cm]{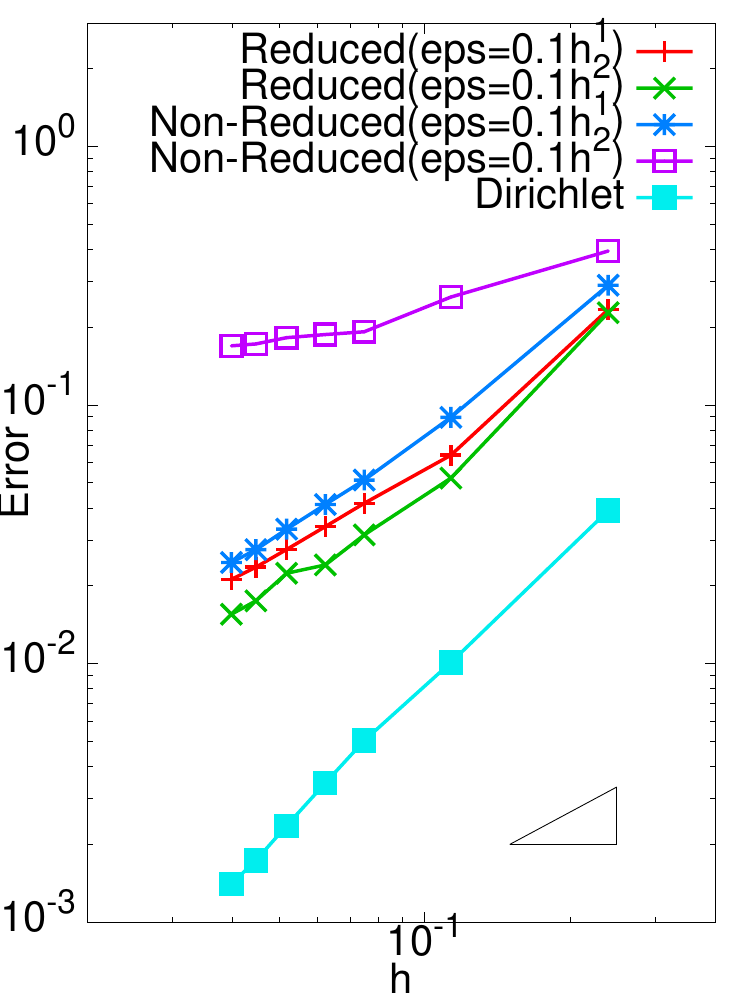} \hspace{1mm}
	\includegraphics[width=5cm]{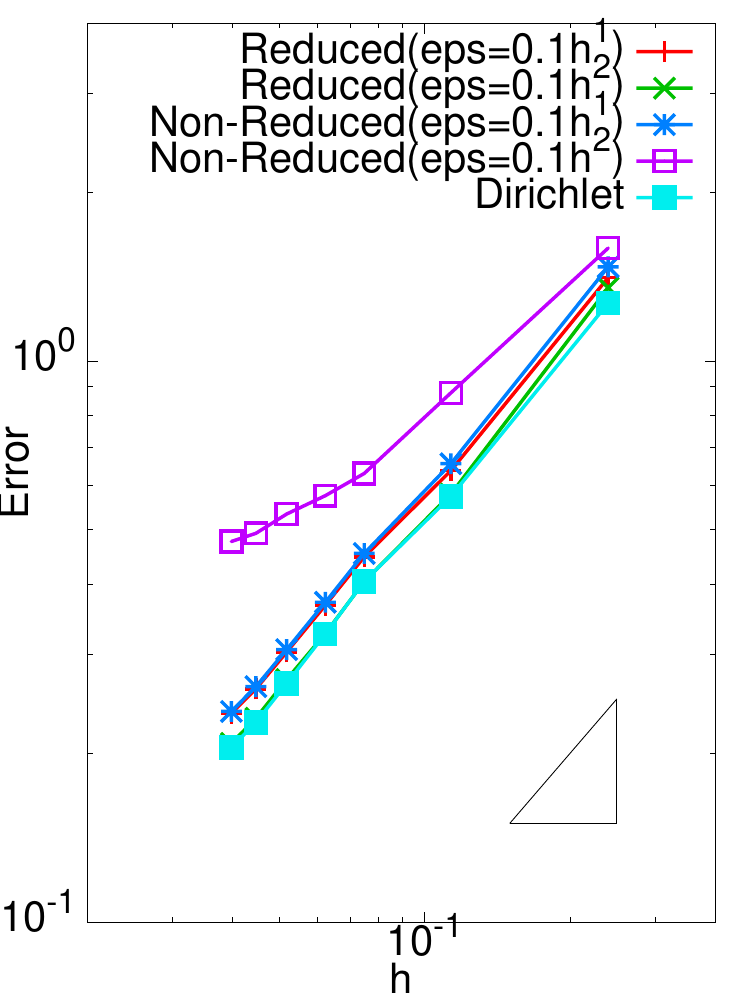} \hspace{1mm}
	\includegraphics[width=5cm]{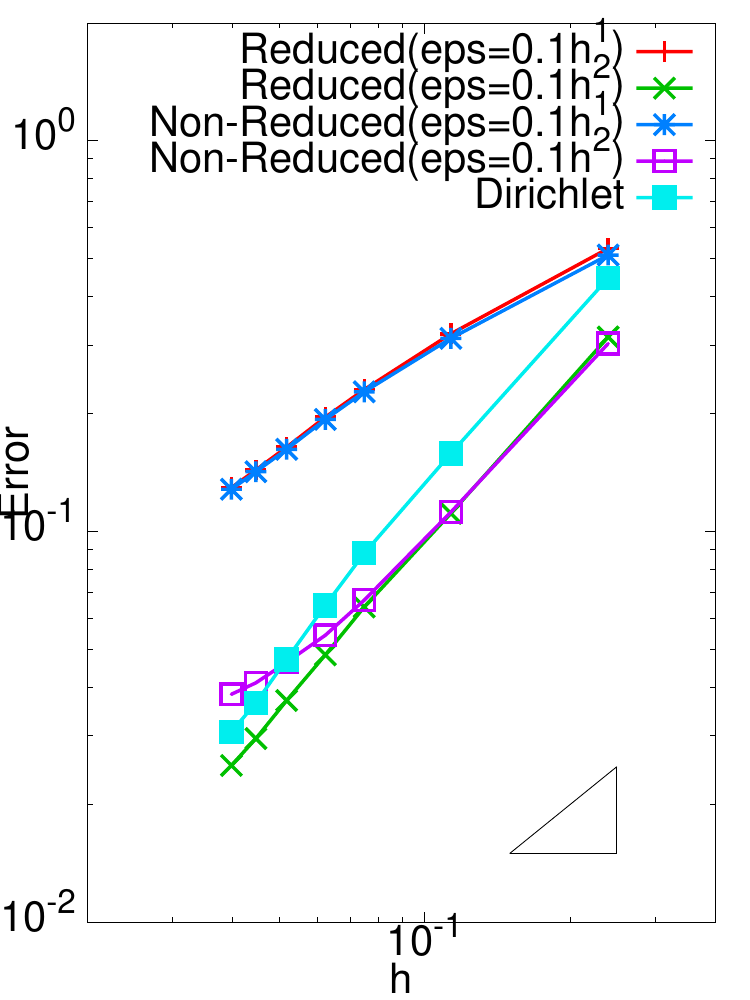}
	\caption{Convergence behavior of $\|u - u_h\|_{L^2(\Omega_h)}$ (left), $\|u - u_h\|_{H^1(\Omega_h)}$ (middle), $\|p - p_h\|_{L^2(\Omega_h)}$ (right) in the 3D test.
	The triangles indicate the slope $O(h)$.}
	\label{fig:convergence behavior in 3D}
\end{figure}

\subsection{Affect of penalty parameter on linear solvers}
It is known that the use of too strong penalty (i.e.\ small $\epsilon$ in our case) would lead to ill-conditioned problems (see e.g.\ \cite{Cas02}), which could deteriorate the performance of linear solvers.
Therefore, we examine a condition number of the matrix $A$ obtained from our penalty FE scheme, in particular, its dependency on the penalty parameter $\epsilon$.
For this purpose, we fix the mesh $h\approx0.113$ in the 3D test, varying $\epsilon$ from 100 to $10^{-8}$.
Moreover, we consider only the reduced scheme since the behavior was similar for the non-reduced one.
The condition number is then estimated by the Matlab function \verb+condest(A)+.
We also report the number of iterations required for GMRES and BiCGSTAB to converge, which were preconditioned by incomplete LU factorization.
The results are presented in \tabref{tab:epsilon and GMRES}.
It seems that the condition number grows at the rate $O(\epsilon^{-2})$, which is faster than the one explained in \cite[p.\ 532]{Cas02}.
One also notices that GMRES and BiCGSTAB failed to converge when $\epsilon<10^{-5}$.
We remark that, even for such small $\epsilon$, sparse direct solvers like UMFPACK or MUMPS were able to solve the linear system (apparently with no problem).
Although we do not have a good explanation of this phenomenon, it tells us that one should carefully choose $\epsilon$ in order to assure both the accuracy and the numerical stability, especially when one wants to invoke iterative methods for solving linear systems obtained from the penalty method.

\begin{table}[htbp]
	\centering
	\caption{Penalty parameter, condition number, and number of iterations for GMRES (the 4th and 5th columns) or BiCGSTAB (the last column) to converge in the 3D test with the mesh $h\approx0.113$ (DOF is 112476).
	The absolute and relative tolerances are set to $10^{-10}$ and $10^{-6}$, respectively.}
	\label{tab:epsilon and GMRES}
	\begin{tabular}{cccccc}
		$\epsilon$ & condest(A) & Rate & Itr (restart=30) & Itr (restart=200) & Itr (BiCGSTAB) \\
		\toprule
		1.0E+2 & 2.36E+6   & ---        & 655   & 182   & 1373 \\
		1.0E+1 & 2.27E+6   & $(<0)$ & 672   & 191   & 165   \\
		1.0E+0 & 2.45E+6   & 0.03     & 733   & 195   & 516   \\
		1.0E-1  & 3.53E+6   & 0.16     & 392   & 165   & 264   \\
		1.0E-2  & 1.64E+7   & 0.67     & 480   & 139   & 152   \\
		1.0E-3  & 1.46E+8   & 0.95     & 539   & 195   & 894   \\
		1.0E-4  & 1.56E+9   & 1.03     & 1432 & 352   & 2888 \\
		1.0E-5  & 1.54E+11 & 2.00     & 28162 & 370 & 2293 \\
		1.0E-6  & 1.54E+13 & 2.00     & (not converged) & (not converged) & (not converged) \\
		1.0E-7  & 1.54E+15 & 2.00     & (not converged) & (not converged) & (not converged) \\
		1.0E-8  & 1.54E+17 & 2.00     & (not converged) & (not converged) & (not converged) \\
	\end{tabular}
\end{table}

\section{Conclusion}
We investigated the P1/P1 and P1b/P1 finite element approximations for the Stokes equations subject to the slip boundary condition in a domain with a smooth boundary.
The constraint $u\cdot n = g$ on $\Gamma$ is relaxed by using the penalty method, which enables us to avoid a variational crime and makes the numerical implementation easier.
We developed a framework to address the difficulty due to $\Omega\neq\Omega_h$ and successfully applied it to establish error estimates of the finite element approximation.
The use of reduced-order numerical integration together with $\epsilon = O(h^2)$ in the penalty term improves the accuracy, which was theoretically justified for $N=2$ and was numerically confirmed for $N=3$.
In fact, we observed that the accuracy in the energy norm was comparable with that of numerical solutions subject to Dirichlet boundary conditions.

\appendix
\section{Transformation between $\Gamma$ and $\Gamma_h$ and related estimates} \label{sec:transformation between Gamma and Gammah}
Let $\Omega$ be a bounded domain in $\mathbb R^N$.
Its boundary $\Gamma$ is assumed to be $C^{1,1}$-smooth, namely, there exist a system of local coordinates $\{(U_r, y_r, \varphi_r)\}_{r=1}^M$ and positive numbers $\alpha, \beta$ such that:
1) $\{U_r\}_{r=1}^M$ forms an open covering of $\Gamma$;
2) $y_r = (y_{r1}, \dots, y_{rN-1}, y_{rN}) = (y_r', y_{rN})$ is a rotated coordinate of the original one $x$, that is, $y_r = A_rx$ for some orthogonal transformation $A_r$; 
3) $\varphi_r\in C^{1,1}(\Delta_r)$ gives a graph representation of $\Gamma\cap U_r$, where $\Delta_r := \{y_r' \in \mathbb R^{N-1} :\; |y_r'| < \alpha\}$, that is,
\begin{align*}
	\Gamma \cap U_r &= \{ y_r \in \mathbb R^N \,:\; y_r' \in \Delta_r \text{ and } y_{rN} = \varphi_r(y_r') \}, \\
	\Omega \cap U_r &= \{ y_r \in \mathbb R^N \,:\; y_r' \in \Delta_r \text{ and } \varphi_r(y_r') < y_{rN} < \varphi_r(y_r') + \beta \}, \\
	\overline\Omega^c \cap U_r &= \{ y_r \in \mathbb R^N \,:\; y_r' \in \Delta_r \text{ and } \varphi_r(y_r') - \beta < y_{rN} < \varphi_r(y_r') \}.
\end{align*}
Since $\Gamma$ is compact, there exists $h_0>0$ such that for any $x\in\Gamma$ the open ball $B(x; h_0)$ is contained in some local coordinate neighborhood $U_r$.
According to the fact that $C^{1,1}(\Delta_r) = W^{2,\infty}(\Delta_r)$, the derivatives of $\varphi_r$ are bounded up to second order, i.e.,
\begin{equation*}
	\|\varphi_r\|_{L^\infty(\Delta_r)} \le C_0, \quad \|\nabla'\varphi_r\|_{L^\infty(\Delta_r)} \le C_1, \quad \|\nabla'^2\varphi_r\|_{L^\infty(\Delta_r)} \le C_2,
\end{equation*}
where $C_0,\,C_1,\,C_2$ are constants independent of $r$, and $\nabla'$ means $\nabla_{y'}$.

The smoothness of $\Gamma$ is connected with that of the signed distance function $d(x)$ defined by
\begin{equation*}
	d(x) =
	\begin{cases}
		-\mathrm{dist}(x, \Gamma) & \text{if }\; x\in\Omega, \\
		\mathrm{dist}(x, \Gamma) & \text{if }\; x\in\Omega^c.
	\end{cases}
\end{equation*}
We collect several known properties on $d(x)$ below. For the details, see e.g.\ \cite[Section 14.6]{GiTr98} or \cite[Section 7.8]{DeZo11}.
Let $\Gamma(\delta) := \{ x\in\mathbb R^N :\; |d(x)| < \delta \}$ be a tubular neighborhood of $\Gamma$ with width $2\delta$.
Then there exists $\delta$ depending only on the curvature of $\Gamma$ such that for arbitrary $x\in\Gamma(\delta)$ the decomposition
\begin{equation} \label{eq:decomposition}
	x = \pi(x) + d(x)\, n(\pi(x)), \quad \pi(x)\in\Gamma,
\end{equation}
is uniquely determined.
Here, $n$ is the outer unit normal field defined on $\Gamma$, which coincides with $\nabla d|_{\Gamma}$.
We extend $n$ from $\Gamma$ to $\Gamma(\delta)$ by $n(x) = \nabla d(x)$, which also agrees with $n(\pi(x))$.
The fact that $\Gamma$ is $C^{1,1}$-smooth implies that $d\in C^{1,1}(\Gamma(\delta))$, $n\in C^{0,1}(\Gamma(\delta))$, and $\pi\in C^{0,1}(\Gamma(\delta))$.
We call $\pi: \Gamma(\delta)\to \Gamma$ the \emph{orthogonal projection onto $\Gamma$}, in view of its geometrical meaning.
We may assume that
\begin{equation*}
	\|d\|_{L^\infty(\Gamma(\delta))} \le C_0, \quad \|\nabla d\|_{L^\infty(\Gamma(\delta))} \le C_1, \quad \|\nabla^2d\|_{L^\infty(\Gamma(\delta))} \le C_2,
\end{equation*}
where we re-choose the constants $C_0,\,C_1,\,C_2$ if necessary.

Now we introduce a regular family of triangulations $\{\mathcal T_h\}_{h\downarrow0}$ of $\overline\Omega$ in the sense of \sref{subsec:triangulation}.
As before, we denote by $\mathcal S_h$ the boundary mesh inherited from $\mathcal T_h$, and we set $\overline\Omega_h = \cup_{T\in\mathcal T_h}T$ and $\Gamma_h = \cup_{S\in\mathcal S_h}S$.
In order for $\Gamma_h$ to be compatible with the local-coordinate system $\{(U_r, y_r, \varphi_r)\}_{r=1}^M$, we assume the following:
\begin{enumerate}[1)]
	\item the mesh size $h$ is less than $\min\{h_0, 1\}$;
	\item for every $r=1,\dots,M$, $\Gamma_h\cap U_r$ is represented by a graph $\{(y_r', \varphi_{rh}(y_r'))\in\mathbb R^N :\; y_r'\in\Delta_r \}$;
	\item every vertex of $S\in\mathcal S_h$ lies on $\Gamma$.
\end{enumerate}
From these we see that every $S\in\mathcal S_h$ is contained in some $U_r$ and that $\varphi_{rh}$ is a piecewise linear interpolation of $\varphi_r$.
As a result of interpolation error estimates, for arbitrary $r$ we obtain
\begin{equation*}
	\begin{aligned}
		\|\varphi_{rh}\|_{L^\infty(\Delta_r)} &\le C_0, \\
		\|\nabla'\varphi_{rh}\|_{L^\infty(\Delta_r)} &\le C_1,
	\end{aligned}
	\qquad
	\begin{aligned}
		\|\varphi_r - \varphi_{rh}\|_{L^\infty(\Delta_r)} &\le C_{0E}h^2, \\
		\|\nabla'(\varphi_r - \varphi_{rh})\|_{L^\infty(\Delta_r)} &\le C_{1E}h,
	\end{aligned}
\end{equation*}
where we re-choose $C_0,\,C_1$ if necessary and the subscript E refers to ``error''.
In the following, $h$ is made small enough to satisfy $2C_{0E}h^2< \min\{h_0, \delta\}$, which in particular ensures that $\pi$ is well-defined on $\Gamma_h$.
We may assume further that $\pi(S)$ is contained in the same $U_r$ that contains $S$.

Based on the observations above, we see that the orthogonal projection $\pi$ maps $\Gamma_h$ into $\Gamma$.
It indeed gives a homeomorphism between $\Gamma_h$ and $\Gamma$, and is element-wisely a diffeomorphism, as shown below.
The representation of a function $f(x)$ in each local coordinate $(U_r, y_r, \varphi_r)$ is defined as $\tilde f(y_r) := f(A_r^{-1}y_r)$.
However, with some abuse of notation, we denote it simply by $f(y_r)$.
Then, the local representation of the outer unit normal $n$ associated to $\Gamma$ is given by
\begin{align*}
	n(y_r', \varphi_r(y_r')) = \frac1{ K_r(y_r') } \begin{pmatrix} \nabla'\varphi_r(y_r') \\ -1 \end{pmatrix}, \qquad y_r'\in\Delta_r,
\end{align*}
where $K_r(y_r') := \sqrt{1 + |\nabla'\varphi_r(y_r')|^2}$.
We are ready to state the following.

\begin{prop} \label{prop:homeo}
	If $h>0$ is sufficiently small, then $\pi|_{\Gamma_h}$ is a homeomorphism between $\Gamma_h$ and $\Gamma$.
\end{prop}
\begin{proof}
	Let us construct an inverse map $\pi^*: \Gamma\to\Gamma_h$ of $\pi|_{\Gamma_h}$ which is continuous.
	To this end, we fix arbitrary $x\in\Gamma$ and choose a local coordinate $(U_r, y_r, \varphi_r)$ of $x$ such that $B(x; h_0)\subset U_r$.
	For simplicity, we omit the subscript $r$ in the following.
	In view of the definition of $\pi(x)$ by \eref{eq:decomposition}, we introduce a segment given by
	\begin{equation*}
		\begin{pmatrix} y' \\ \varphi(y') \end{pmatrix} + \frac{t}{K(y')} \begin{pmatrix} \nabla'\varphi(y') \\ -1 \end{pmatrix}, \qquad |t|\le 2C_{0E}h^2.
	\end{equation*}
	To each point on this segment we associate its height $H(t)$ with respect to the graph of $\varphi_h$, that is,
	\begin{equation*}
		H(t) = \varphi(y') - \frac{t}{K(y')} - \varphi_h\left( y' +  \frac{t}{K(y')} \nabla'\varphi(y') \right).
	\end{equation*}
	Then we assert that $\frac{d}{dt}H(t) < 0$ and that $H(-2C_{0E}h^2) > 0$, $H(2C_{0E}h^2) < 0$.
	
	To prove the first assertion, letting
	\begin{equation*}
		Y' := y' +  \frac{t}{K(y')} \nabla'\varphi(y'), 
	\end{equation*}
	we have $\frac{d}{dt}H(t) = -\frac1{K(y')}(1 + \nabla'\varphi(y')\cdot \nabla'\varphi_h(Y'))$.
	One sees that
	\begin{align*}
		\nabla'\varphi_h(Y') &= \nabla'\varphi(y') - \nabla'\varphi(y') + \nabla'\varphi(Y') - \nabla'\varphi(Y') + \nabla'\varphi_h(Y') \\
			&=: \nabla'\varphi(y') + I_1 + I_2,
	\end{align*}
	where $I_1$ and $I_2$ satisfy
	\begin{equation*}
		|I_1| \le \|\nabla'^2\varphi\|_{L^\infty(\Delta)} |Y' - y'| \le C_2\cdot 2C_{0E}h^2, \qquad |I_2| \le C_{1E}h.
	\end{equation*} 
	Then it follows that
	\begin{equation*}
		1 + \nabla'\varphi(y')\cdot \nabla'\varphi_h(Y') \ge K(y')^2 - C_1(2C_2C_{0E}h^2 + C_{1E}h).
	\end{equation*}
	From this we have $\frac{d}{dt}H(t) < 0$ provided $C_1(2C_2C_{0E}h^2 + C_{1E}h) < 1/2$.
	For the second assertion, one finds that
	\begin{align*}
		H(-2C_{0E}h^2) &= \frac{2C_{0E}h^2}{K(y')} + \varphi(y') - \varphi_h(Y') \\
			&= \frac{2C_{0E}h^2}{K(y')} + \varphi(y') - \varphi(Y') + \varphi(Y') - \varphi_h(Y'),
	\end{align*}
	where $Y'$ is $y' -  \frac{2C_{0E}h^2}{K(y')} \nabla'\varphi(y')$.
	By Taylor's theorem, there exists some $\theta\in(0,1)$ such that
	\begin{align*}
		\varphi(Y') - \varphi(y') &= \nabla'\varphi(y')\cdot(Y' - y') + \frac12 (Y' - y')^T \nabla'^{2}\varphi|_{y' + \theta(Y' - y')} (Y' - y') \\
			&\le - \frac{2C_{0E}h^2}{K(y')} |\nabla'\varphi(y')|^2 + 2 C_2C_{0E}^2 h^4.
	\end{align*}
	By the definition of $K(y)$ and by $|\varphi(Y') - \varphi_h(Y')| \le C_{0E}h^2$, we obtain
	\begin{align*}
		H(-2C_{0E}h^2) &\ge 2K(y')C_{0E}h^2 - 2C_2C_{0E}^2 h^4 - C_{0E} h^2 \\
			&\ge 2C_{0E}h^2 - 2C_2C_{0E}^2 h^4 - C_{0E} h^2 = C_{0E}h^2(1 - 2C_2C_{0E}h^2),
	\end{align*}
	which implies that $H(-2C_{0E}h^2) > 0$ provided $C_2C_{0E}h^2 \le 1/4$.
	In the same way, the last assertion $H(2C_{0E}h^2) < 0$ can be proved.
	
	From these assertions we deduce that there exists a unique $t^*(x) \in[-2C_{0E}h^2, 2C_{0E}h^2]$ such that $H(t^*(x)) = 0$.
	Consequently, the map $\pi^*: \Gamma\to\Gamma_h; \; x\mapsto x + t^*(x)n(x)$ is well-defined.
	A direct computation combined with the uniqueness of the decomposition \eref{eq:decomposition} shows that $\pi^*$ is the inverse of $\pi|_{\Gamma_h}$.
	The continuity of $\pi^*$, especially that of $t^*$, follows from an argument similar to the proof of the implicit function theorem (see e.g.\ \cite[Theorem 3.2.1]{KrPa02}).
\end{proof}

\pref{prop:homeo} enables us to define an \emph{exact triangulation of $\Gamma$} by
\begin{equation*}
	\pi(\mathcal S_h) = \{ \pi(S) \,:\, S\in\mathcal S_h \}.
\end{equation*}
In particular, we can subdivide $\Gamma$ into disjoint sets as $\Gamma = \bigcup_{S\in\mathcal S_h}\pi(S)$.
Furthermore, for each $S\in\mathcal S_h$ we see that $S$ and $\pi(S)$ admit the same domain of parametrization, which is important in the subsequent analysis.
To describe this fact, we choose a local coordinate $(U_r, y_r, \varphi_r)$ such that $U_r \supset S\cup\pi(S)$, and introduce the \emph{projection to the base set} $b_r:\mathbb R^N\to\mathbb R^{N-1}$ by $b_r(y_r) = y_r'$.
The domain of parametrization is then defined to be $S' = b_r(\pi(S))$.
We observe that the mappings
\begin{equation*}
	\begin{matrix*}[l]
		&\Phi: S' \to \pi(S); & &y_r' \mapsto (y_r', \varphi_{r}(y_r'))^T, \\
		&\Phi_h: S' \to S;   & &y_r' \mapsto \pi^*(y_r', \varphi_{r}(y_r')) = \Phi(y_r') + t^*(y_r')\, n(\Phi(y_r')),
	\end{matrix*}
\end{equation*}
are bijective and that $\Phi$ is smooth on $S'$.
If in addition $\Phi_h$, especially $t^*$, is also smooth on $S'$, then $\Phi$ and $\Phi$ may be employed as smooth parametrizations for $S$ and $\pi(S)$ respectively.
The next proposition verifies that this is indeed the case.

\begin{prop} \label{prop:estimate of Phi}
	Under the setting above, we have
	\begin{equation*}
		\|t\|_{L^\infty(S')} \le \tilde C_{0E} h^2, \qquad  \|\nabla' t\|_{L^\infty(S')} \le \tilde C_{1E} h,
	\end{equation*}
	where $\tilde C_{0E}$ and $\tilde C_{1E}$ are constants depending only on $N$ and $\Gamma$.
\end{prop}
\begin{proof}
	Since the first relation is already obtained in \pref{prop:homeo} with $\tilde C_{0E} = 2C_{0E}$, we focus on proving the second one.
	For notational simplicity, we omit the subscript $r$ and also use the abbreviation $\partial_i = \frac{\partial}{\partial y_i}\,(i=1,\dots,N)$.
	The fact that $t^*$ is differentiable with respect to $y'$ can be shown in a way similar to the proof of the implicit function theorem.
	Thereby it remains to evaluate the supremum norm of $\nabla' t$ in $S'$, which we address in the following.
	
	Recall that $t^*(y')$ is determined according to the equation
	\begin{equation} \label{eq:tilde t}
		\tilde t(y') = \varphi(y') - \varphi_h(y' + \tilde t(y') \nabla'\varphi(y')),
	\end{equation}
	where we have set $\tilde t(y') := t^*(y')/K(y')$. Because $\nabla't^* = K\nabla'\tilde t + \frac{\nabla'^2\varphi \nabla'\varphi}{K}\, t^*$, it follows that
	\begin{equation*}
		\|\nabla't^*\|_{L^\infty(S')} \le (1 + C_1) \|\nabla'\tilde t\|_{L^\infty(S')} + C_2C_1\tilde C_{0E}h^2.
	\end{equation*}
	Therefore, it suffices to prove that $\|\nabla'\tilde t\|_{L^\infty(S')} \le Ch$; here and hereafter $C$ denotes various constants which depends only on $N$ and $\Gamma$.
	
	Applying $\nabla'$ to \eref{eq:tilde t} gives
	\begin{equation} \label{eq:nabla prime tilde t}
		\big( 1 + \nabla'\varphi(y')\cdot \nabla'\varphi_h(Y') \big) \nabla'\tilde t = \nabla'\varphi(y') - \nabla'\varphi_h(Y') + \tilde t(y') \nabla'^2\varphi(y') \nabla'\varphi_h(Y'),
	\end{equation}
	where $Y' := y' + \tilde t(y') \nabla'\varphi(y')$.
	By the same way as we estimated $I_1$ and $I_2$ in the proof of \pref{prop:homeo}, we obtain
	\begin{align*}
		|\nabla'\varphi(y') - \nabla'\varphi_h(Y')| &\le C_2\tilde C_{0E}h^2 + C_{1E}h \le Ch, \\
		1 + \nabla'\varphi(y')\cdot\nabla'\varphi_h(Y') &\ge K(y')^2 - C_1(C_2\tilde C_{0E}h^2 + C_{1E}h) \ge \frac12.
	\end{align*}
	Also we see that
	\begin{equation*}
		|\tilde t(y') \nabla'^2\varphi(y') \nabla'\varphi_h(Y')| \le \tilde C_{0E}h^2\cdot C_2C_1 \le Ch^2.
	\end{equation*}
	Combining these observations with \eref{eq:nabla prime tilde t}, we deduce the desired estimate $\|\nabla'\tilde t\|_{L^\infty(S')} \le Ch$.
\end{proof}
\begin{rem}
	Let $\Gamma\in C^{2,1}$.
	Since $\varphi_h$ is linear on $b(S)$, further differentiation of \eref{eq:nabla prime tilde t} gives us $t^*\in C^{1,1}(\pi(S))$; in fact we have $\|\nabla'^2t\|_{L^\infty(S')} \le \tilde C_{2E}$.
	This implies that $\pi|_{S}$ is a $C^{1,1}$-diffeomorphism between $S$ and $\pi(S)$.
	However, since $\nabla'\phi_h$ is smooth only within $b(S)$, $\pi$ is not globally a diffeomorphism.
\end{rem}

Now we give an error estimate for surface integrals on $\Gamma$ and $\Gamma_h$.
Heuristically speaking, the result reads $|d\gamma - d\gamma_h| \le O(h^2)$, which may be found in the literature (see e.g.\ \cite{Dzi88}).
Here and hereafter, we denote the surface elements of $\Gamma$ and $\Gamma_h$ by $d\gamma$ and $d\gamma_h$, respectively.
\begin{thm} \label{thm:error estimate of surface integral}
	Let $S\in\mathcal S_h$ and $f$ be an integrable function on $S$.
	Then we have
	\begin{equation*}
		\left| \int_{\pi(S)}f\,d\gamma - \int_S f\circ\pi\, d\gamma_h \right| \le Ch^2 \int_S |f|\,d\gamma,
	\end{equation*}
	where $C$ is a constant depending only on $N$ and $\Gamma$.
\end{thm}
\begin{proof}
	Let $(U_r, y_r, \varphi_r)$ be a local coordinate that contains $S\cup\pi(S)$.
	We omit the subscript $r$ and use the abbreviation $\partial_i = \frac{\partial}{\partial y_i}\,(i=1,\dots,N)$.
	We represent the surface integral using the parametrization $\Phi$ as follows:
	\begin{equation*}
		\int_{\pi(S)}f\,d\gamma = \int_{S'} f(\Phi(y')) \sqrt{\mathrm{det}\,G}\,dy',
	\end{equation*}
	where $G=(G_{ij})_{1\le i,j\le N-1}$ denotes the Riemannian metric tensor given by $G_{ij} = \partial_i\Phi \cdot \partial_j\Phi$ (dot means the inner product in $\mathbb R^N$).
	Similarly, noting that $\pi\circ\Phi_h = \Phi$, one obtains
	\begin{equation*}
		\int_{S}f\circ\pi\,d\gamma_h = \int_{S'} f(\Phi(y')) \sqrt{\mathrm{det}\,G_h}\,dy',
	\end{equation*}
	where $G_h$ is given by $G_{h,ij} = \partial_i\Phi_h \cdot \partial_j\Phi_h$. Then we assert that:
	\begin{equation*}
		\|G_h - G\|_{L^\infty(S')} \le Ch^2.
	\end{equation*}
	
	To prove this, noting that $\Phi_h = \Phi + t^*n\circ\Phi$, we compute each component of $G_h-G$ as follows:
	\begin{align*}
		G_{h,ij} - G_{ij} &= \partial_i\Phi \cdot \partial_j(\Phi_h - \Phi) + \partial_j\Phi \cdot \partial_i(\Phi_h - \Phi) + \partial_i(\Phi_h - \Phi) \cdot \partial_j(\Phi_h - \Phi) \\
					&= \partial_i\Phi \cdot \partial_j(t^*n\circ\Phi) + \partial_j\Phi \cdot \partial_i(t^*n\circ\Phi) + \partial_i(t^*n\circ\Phi) \cdot \partial_j(t^*n\circ\Phi) \\
					&=: I_1 + I_2 + I_3.
	\end{align*}
	For $I_1$, we notice that $\partial_i\Phi$ is a tangent vector so that $\partial_i\Phi\cdot n\circ\Phi=0$. This yields
	\begin{equation*}
		I_1 = \partial_i\Phi\cdot t^*\partial_j(n\circ\Phi) = t^*\partial_i\Phi\cdot (\partial_jn + \partial_j\varphi\,\partial_Nn)|_{\Phi},
	\end{equation*}
	which is estimated by $\tilde C_{0E}h^2(1 + C_1)(C_2 + C_1C_2)$ thanks to \pref{prop:estimate of Phi}.
	$I_2$ can be bounded in the same manner.
	To estimate $I_3$, we observe that
	\begin{equation*}
		\partial_i(t^*n\circ\Phi) = (\partial_i t^*) n\circ\Phi + t^* \big( \partial_in + \partial_i\varphi\,\partial_Nn \big)|_{\Phi},
	\end{equation*}
	which is bounded by $\tilde C_{1E}h + \tilde C_{0E}h^2(C_2 + C_1C_2) \le Ch$.
	Similarly one gets $|\partial_j(t^*n\circ\Phi)| \le Ch$, hence it follows that $|I_3| \le Ch^2$.
	Therefore, $|G_{h,ij} - G_{ij}| \le Ch^2$, which proves the assertion.
	
	Now we use the following crude estimate for perturbation of determinants (cf.\ \cite[equation (3.13)]{Kno96}): if $A$ and $B$ are $N\times N$ matrices such that $|A_{ij}| \le a$ and $|B_{ij}|\le b$ for all $i,j$, then
	\begin{equation*}
		|\mathrm{det}\,(A+B) - \mathrm{det}\,A| \le N!N(a+b)^{N-1}b.
	\end{equation*}
	Combining this with the assertion above and also with $\sqrt\alpha - \sqrt\beta = (\alpha - \beta)/(\sqrt\alpha + \sqrt\beta)$, we obtain
	\begin{equation*}
		\|\sqrt{\mathrm{det}\,G} - \sqrt{\mathrm{det}\,G_h}\|_{L^\infty(S')} \le Ch^2.
	\end{equation*}
	In addition, note that $\sqrt{\mathrm{det}\,G} = \sqrt{1 + |\nabla'\varphi|^2} \ge 1$.
	Consequently,
	\begin{align*}
		\left| \int_{\pi(S)}f\,d\gamma - \int_S f\circ\pi\, d\gamma_h \right| \le Ch^2\,\int_{S'} |f(\Phi(y'))| \sqrt{\mathrm{det}\,G} \,dy' = Ch^2 \int_{\pi(S)}|f|\,d\gamma,
	\end{align*}
	which proves the theorem.
\end{proof}
\begin{rem}
	Adding up the results of the theorem for all $S\in\mathcal S_h$ yields
	\begin{equation*}
		\left| \int_\Gamma f\,d\gamma - \int_{\Gamma_h} f\circ\pi\, d\gamma_h \right| \le Ch^2 \int_\Gamma |f|\,d\gamma.
	\end{equation*}
	It also follows that $|\int_{\Gamma_h} f\circ\pi\, d\gamma_h| \le C\int_\Gamma |f|\,d\gamma$.
	Choosing in particular $|f|^p$ as the integrand gives $\|f\circ\pi\|_{L^p(\Gamma_h)} \le C^{1/p}\|f\|_{L^p(\Gamma)}$ for $p\in[1,\infty]$.
\end{rem}

Let $f$ be a smooth function given on $\Gamma$.
Then its transformation to $\Gamma_h$ is defined by $f\circ\pi$.
However, if $f$ is extended to a neighborhood of $\Gamma$, e.g.\ to $\Gamma(\delta)$, then we may also consider $f$'s natural trace on $\Gamma_h$.
The next theorem provides error estimation of these two quantities.

\begin{thm} \label{thm:trace and transform}
	Let $f\in W^{1,p}(\Gamma(\delta_1))$, where $p\in[1,\infty]$ and $\delta_1 \in [\tilde C_{0E}h^2, 2\tilde C_{0E}h^2]$. Then,
	\begin{equation*}
		\|f - f\circ\pi\|_{L^p(\Gamma_h)} \le C\delta_1^{1-1/p}\|f\|_{W^{1,p}(\Gamma(\delta_1))},
	\end{equation*}
	where $C$ is a constant depending only on $p,\,N,\,\Omega$.
\end{thm}
\begin{proof}
	Since $\Gamma(\delta_1) = \bigcup_{S\in\mathcal S_h} \pi(S,\delta_1)$, where $\pi(S, \delta_1) = \{ x\in\Gamma(\delta_1) : \pi(x)\in\pi(S) \}$ denotes a tubular neighborhood of $S$, it suffices to prove that
	\begin{equation} \label{eq:error of trace and transform on S}
		\int_S |f - f\circ\pi|^p\,d\gamma \le C\delta_1^{p - 1} \int_{\pi(S, \delta_1)} |\nabla f|^p\,dx \qquad \forall S\in\mathcal S_h.
	\end{equation}
	To this end, using the notation in \tref{thm:error estimate of surface integral}, we estimate the left-hand side of \eref{eq:error of trace and transform on S} by
	\begin{align*}
		\int_S |f - f\circ\pi|^p\,d\gamma &= \int_{S'} |f\circ\Phi_h - f\circ\Phi|^p \sqrt{\mathrm{det}\,G_h}\,dy' \\
			&\le C \int_{S'} |f\circ\Phi_h - f\circ\Phi|^p\,dy'.
	\end{align*}
	Here, for fixed $y'\in S'$ we have
	\begin{align*}
		f(\Phi_h(y')) - f(\Phi(y')) &= \int_0^1 \frac{d}{ds} f\big( \Phi(y') + s(\Phi_h(y') - \Phi(y')) \big)\,ds \\
					            &= \int_0^1 \frac{d}{ds} f\big( \Phi(y') + s\, t^*(y') n(\Phi(y')) \big)\,ds \\
					            &= \int_0^1 t^*(y') n(\Phi(y')) \cdot \nabla f\big( \Phi(y') + s\, t^*(y') n(\Phi(y')) \big)\,ds \\
					            &= \int_0^{t^*(y')} n(\Phi(y')) \cdot \nabla f\big( \Phi(y') + t n(\Phi(y')) \big)\,dt.
	\end{align*}
	Because $|t^*(y')| \le \tilde C_{0E}h^2 \le \delta_1$, it follows that
	\begin{align*}
		|f(\Phi_h(y')) - f(\Phi(y'))| &\le \int_{-\delta_1}^{\delta_1} \big| \nabla f\big( \Phi(y') + tn(\Phi(y')) \big) \big|\,dt \\
		                                       &\le (2\delta_1)^{1-1/p} \left(  \int_{-\delta_1}^{\delta_1} \big| \nabla f\big( \Phi(y') + tn(\Phi(y')) \big) \big|^p\,dt \right)^{1/p},
	\end{align*}
	where we have used H\"older's inequality. Consequently,
	\begin{equation} \label{eq:remove process}
		\int_S |f - f\circ\pi|^p\,d\gamma \le C\delta_1^{p-1} \int_{S'\times[-\delta_1,\delta_1]} \big| \nabla f\big( \Phi(y') + tn(\Phi(y')) \big) \big|^p\,dy'dt.
	\end{equation}
	
	On the other hand, we observe that the $N$-dimensional transformation
	\begin{equation} \label{eq:Psi}
		\Psi: S'\times[-\delta_1, \delta_1] \to \pi(S, \delta_1); \quad (y', t) \mapsto \Phi(y') + tn(\Phi(y'))
	\end{equation}
	is bijective and smooth.
	Application of this transformation to the right-hand side of \eref{eq:error of trace and transform on S} leads to
	\begin{equation*}
		\int_{\pi(S, \delta_1)} |\nabla f|^p\,dx = \int_{S'\times[-\delta_1, \delta_1]} \big| \nabla f\big( \Phi(y') + tn(\Phi(y')) \big) \big|^p\, |\mathrm{det}\, J| \,dy'dt,
	\end{equation*}
	where $J = (\partial_1\Phi + t\partial_1(n\circ\Phi), \cdots, \partial_{N-1}\Phi + t\partial_{N-1}(n\circ\Phi), n\circ\Phi)$ denotes the Jacobi matrix of $\Psi$.
	Letting $\tilde J := (\partial_1\Phi, \cdots, \partial_{N-1}\Phi, n\circ\Phi)$, we find that
	\begin{equation*}
		\|J - \tilde J\|_{L^\infty(S'\times[-\delta_1, \delta_1])} \le C\delta_1,
	\end{equation*}
	because $|t\partial_i(n\circ\Phi)| = |t (\partial_in + \partial_i\varphi \partial_Nn)_{|\Phi}| \le \delta_1(C_2 + C_1C_2)$.
	This implies
	\begin{equation*}
		\|\mathrm{det}\, J - \mathrm{det}\,\tilde J\|_{L^\infty(S'\times[-\delta_1, \delta_1])} \le C\delta_1,
	\end{equation*}
	which combined with $\mathrm{det}\,\tilde J = K(y') \ge 1$ yields $\mathrm{det}\,J \ge 1/2$ if $h$ is sufficiently small. Therefore,
	\begin{equation} \label{eq:restore process}
		\int_{\pi(S, \delta_1)} |\nabla f|^p\,dx \ge \frac12 \int_{S'\times[-\delta_1, \delta_1]} \big| \nabla f\big( \Phi(y') + tn(\Phi(y')) \big) \big|^p \,dy'dt.
	\end{equation}
	The desired estimate \eref{eq:error of trace and transform on S} is now a consequence of \eref{eq:remove process} and \eref{eq:restore process}. This completes the proof.
\end{proof}

Finally, we show that the $L^p$-norm in a tubular neighborhood can be bounded in terms of its width.
Such estimate is stated e.g.\ in \cite[Lemma 2.1]{Zha08} or in \cite[equation (3.6)]{Tab01}.
However, since we could not find a full proof of this fact (especially for $N=3$) in the literature, we present it here.

\begin{thm} \label{thm:Lp norm in tubular neighborhood}
	Under the same assumptions as in \tref{thm:trace and transform}, we have
	\begin{equation*}
		\|f\|_{L^p(\Gamma(\delta_1))} \le C (\delta_1 \|\nabla f\|_{L^p(\Gamma(\delta_1))} + \delta_1^{1/p} \|f\|_{L^p(\Gamma)}),
	\end{equation*}
	where $C$ is a constant depending only on $p,\,N,\,\Omega$.
\end{thm}
\begin{proof}
	We adopt the same notation as in the proofs of Theorems \ref{thm:error estimate of surface integral} and \ref{thm:trace and transform}.
	Then it suffices to prove that
	\begin{equation*}
		\int_{\pi(S,\delta_1)} |f|^p\,dy \le C \left( \delta_1^p \int_{\pi(S,\delta_1)} |\nabla f|^p\,dy + \delta_1 \int_{\pi(S)} |f|^p\,d\gamma \right) \qquad \forall S\in\mathcal S_h.
	\end{equation*}
	To this end, using the transformation $\Psi$ given in \eref{eq:Psi} we express the left-hand side as
	\begin{align*}
		\int_{\pi(S,\delta_1)} |f|^p\,dy &= \int_{S'\times [-\delta_1,\delta_1]} |f(\Psi(y', t))|^p|\mathrm{det}\,J(y',t)|\,dy'dt \\
			&\le C \int_{S'\times [-\delta_1,\delta_1]} \big( |f(\Psi(y', t)) - f(\Phi(y'))|^p + |f(\Phi(y'))|^p \big) \,dy'dt \\
			&=: I_1 + I_2.
	\end{align*}
	For $I_1$, we see from the same argument as before that
	\begin{equation*}
		|f(\Psi(y', t)) - f(\Phi(y'))|^p \le (2\delta_1)^{p-1} \int_{-\delta_1}^{\delta_1} \big| \nabla f\big( \Phi(y') + sn(\Phi(y')) \big) \big|^p\,ds,
	\end{equation*}
	which yields
	\begin{equation*}
		|I_1| \le C\delta_1^p \int_{S'\times [-\delta_1,\delta_1]} \big| \nabla f(\Psi(y',s)) \big|^p |\mathrm{det}\,J(y',s)|\,dy'ds = C\delta_1^p \int_{\pi(S,\delta_1)} |\nabla f|^p\,dy.
	\end{equation*}
	For $I_2$, it follows that
	\begin{equation*}
		|I_2| = 2C\delta_1 \int_{S'} |f(\Phi(y'))|^p\,dy' \le 2C\delta_1 \int_{S'} |f(\Phi(y'))|^p \sqrt{\mathrm{det}\,G}\,dy' = 2C\delta_1 \int_{\pi(S)} |f|^p\,d\gamma.
	\end{equation*}
	We have thus obtained the desired estimate, which completes the proof.
\end{proof}

\section{Error of $n$ and $n_h$} \label{sec:error of n and nh}
Let us prove that $|n\circ\pi - n_h| \le O(h)$ on $\Gamma_h$ and also that, when $N=2$, it is improved to $O(h^2)$ if the consideration is restricted to the midpoint of edges.
\begin{lem} \label{lem:n and nh}
	Let $n$ and $n_h$ be the outer unit normals to $\Gamma$ and $\Gamma_h$ respectively. Then there holds
	\begin{equation} \label{eq:error between n and nh}
		\|n\circ\pi - n_h\|_{L^\infty(\Gamma_h)} \le Ch.
	\end{equation}
	If in addition $N=2$, $\Gamma\in C^{2,1}$, and $m_S$ denotes the midpoint of $S\in\mathcal S_h$, then
	\begin{equation} \label{eq:error between n and nh, 2D}
		\sup_{S\in\mathcal S_h}|n\circ\pi(m_S) - n_h(m_S)| \le Ch^2.
	\end{equation}
	Here, $C$ is a constant depending only on $N$ and $\Gamma$.
\end{lem}
\begin{proof}
	Let $S\in\mathcal S_h$ be arbitrary and let $(U_r, y_r, \phi_r)$ be a local coordinate that contains $S\cup\pi(S)$.
	We omit the subscript $r$ in the following.
	One sees that $n$ and $n_h$ are represented as
	\begin{equation*}
		n(y', \varphi(y')) = \frac1{\sqrt{1 + |\nabla'\varphi|^2}} \begin{pmatrix}\nabla'\varphi \\ -1\end{pmatrix}, \quad
		n_h(y', \varphi_h(y')) = \frac1{\sqrt{1 + |\nabla'\varphi_h|^2}} \begin{pmatrix}\nabla'\varphi_h \\ -1\end{pmatrix}, \quad y'\in\Delta.
	\end{equation*}
	A direct computation gives
	\begin{equation} \label{eq:n - nh by graph representation}
		|n(y', \varphi(y')) - n_h(y', \varphi_h(y'))| \le 2|\nabla'(\varphi(y') - \varphi_h(y'))| \le 2C_{1E}h.
	\end{equation}
	This combined with the observation that
	\begin{align*}
		|n\circ\pi(y', \varphi_h(y')) - n(y', \varphi(y'))| &= |n\circ\pi(y', \varphi_h(y')) - n\circ\pi(y', \varphi(y'))| \\
			&\le C_1|\pi(y', \varphi_h(y')) - \pi(y', \varphi(y'))| \\
			&\le C_1 \|\nabla\pi\|_{L^\infty(\Gamma(\delta))} |\varphi_h(y') - \varphi(y')| \\
			&\le C_1 \|\nabla\pi\|_{L^\infty(\Gamma(\delta))} C_{0E}h^2
	\end{align*}
	proves \eref{eq:error between n and nh}.
	When $N=2$ and $\Gamma\in C^{2,1}$, by using Taylor expansion, we find that $m_S$ is a point of super-convergence such that $|\frac{d\varphi_h}{dy_1}(b(m_S)) - \frac{d\varphi}{dy_1}(b(m_S))| \le Ch^2$.
	This improves \eref{eq:n - nh by graph representation} to $O(h^2)$, and thus \eref{eq:error between n and nh, 2D} is proved.
\end{proof}
\begin{rem}
	In the case $N=3$, it is known that the barycenter of a triangle is not a point of super-convergence for the derivative of linear interpolations; see \cite[p.\ 1930]{WCH10}.
	For this reason, \eref{eq:error between n and nh, 2D} holds only for $N=2$.
\end{rem}

\section*{Acknowledgements}
The authors thank Professor Fumio Kikuchi, Professor Norikazu Saito, and Professor Masahisa Tabata for valuable comments to the results of this paper.
They also thank Professor Xuefeng Liu for giving them a program which converts matrices described by the PETSc-format into those by the CRS-format.
The first author was supported by JST, CREST.
The second author was supported by JSPS KAKENHI Grant Number 24224004, 26800089.
The third author was supported by JST, CREST and by JSPS KAKENHI Grant Number 23340023.


\end{document}